\documentclass[11pt,reqno]{NumPDEsArticle}

\usepackage{blindtext}

\usepackage[T1]{fontenc}
\usepackage[english]{babel}
\usepackage{csquotes}
\usepackage{fullpage}
\usepackage{enumitem}
\usepackage{amsmath,amssymb,mathrsfs}
\usepackage{mathtools}
\usepackage{xspace}
\usepackage{hyperref}
\usepackage{NumPDEsMacros}
\usepackage{booktabs,tabularx}
\usepackage{colortbl}

\usepackage[caption=false]{subfig}
\usepackage{subcaption}
\usepackage[export]{adjustbox}

\usepackage{csquotes}
\usepackage{enumitem}
\usepackage{amsmath,amssymb}
\usepackage{orcidlink}
\usepackage{diagbox,graphicx}
\usepackage{stmaryrd}

\newcommand{\AS}{\textnormal{AS}}
\newcommand{\MG}{\textnormal{MG}}
\newcommand{\SMG}{\textnormal{sMG}}
\newcommand{\LMG}{\textnormal{nsMG}}

\newcommand{\fL}{\mathscr{L}}
\newcommand{\Cloc}{\const{C}{loc}^{(1)}}
\newcommand{\CLoc}{\const{C}{loc}^{(2)}}
\newcommand{\Cspan}{\const{C}{span}}
\newcommand{\Cscs}{\const{C}{SCS}}
\newcommand{\Clev}{\const{C}{lev}}

\title{Generalized preconditioned conjugate gradients \\ for adaptive FEM with optimal complexity}

\author{Paula Hilbert\orcidlink{0009-0005-0105-1066}}

\author{Ani Miraçi\orcidlink{0000-0003-4962-9662}}
\address{Sorbonne Université, Université Paris Cité, CNRS, Laboratoire Jacques-Louis Lions, LJLL, F-75005 Paris, France}
\email{\tt ani.miraci@sorbonne-universite.fr}

\author{Dirk Praetorius\orcidlink{0000-0002-1977-9830}}
\address{TU Wien, Institute of Analysis and Scientific Computing, Wiedner Hauptstr. 8--10/E101/4, 1040 Vienna, Austria}

\email{\tt paula.hilbert@asc.tuwien.ac.at \quad {\textrm{(corresponding author)}}}
\email{\tt dirk.praetorius@asc.tuwien.ac.at}



\thanks{This research was funded in whole or in part by the Austrian Science Fund (FWF) projects
\href{https://www.fwf.ac.at/en/research-radar/10.55776/PAT3699424}{10.55776/PAT3699424}
(standalone project PAT3699424 ``Optimal robust solvers for reliable and efficient AFEMs''),
and
\href{https://www.fwf.ac.at/en/research-radar/10.55776/PAT3446525}{10.55776/PAT3446525}
(standalone project PAT3446525 ``Adaptive Uzawa-type FEM for nonlinear PDEs'').
Additionally, Paula Hilbert is supported by the Vienna School of Mathematics.}

\begin{document}

\maketitle

\begin{abstract}
We consider adaptive finite element methods (AFEMs) with inexact algebraic solvers for second-order symmetric linear elliptic diffusion problems.
Optimal complexity of AFEM, i.e., optimal convergence rates with respect to the overall computational cost, hinges on two requirements on the solver. First, each solver step is of linear cost with respect to the number of degrees of freedom. 
Second, each solver step guarantees uniform contraction of the solver error with respect to the PDE-related energy norm. 
Both properties must be ensured robustly with respect to the local mesh size $h$ (i.e., $h$-robustness). 
While existing literature on geometric multigrid methods (MG) or symmetric additive Schwarz preconditioners for the preconditioned conjugate gradient method (PCG) that are appropriately adapted to adaptive mesh-refinement satisfy these requirements, this paper aims to consider more general solvers.
Our main focus is on preconditioners stemming from contractive solvers which need not be symmetrized to be used with Krylov methods and which are not only $h$-robust but also $p$-robust, i.e., the contraction constant is independent of the polynomial degree $p$.
In particular, we show that generalized PCG (GPCG) with an $h$- and $p$-robust contractive MG as a preconditioner satisfies the requirements for optimal-complexity AFEM and that it numerically outperforms AFEM using MG as a solver. 
While this is certainly known for (quasi-)uniform meshes, the main contribution of the present work is the rigorous analysis of the interplay of the solver with adaptive mesh-refinement.
Numerical experiments underline the theoretical findings.
\end{abstract}

\section{Introduction}

Given a bounded Lipschitz domain $\Omega \subset \R^d$ for $d \in \N$, a right-hand side $f \in L^2(\Omega)$, and a  symmetric and uniformly positive definite diffusion coefficient $\vec{K} \in [L^{\infty}(\Omega)]^{d\times d}_{\text{sym}}$, we consider the second-order symmetric linear elliptic diffusion problem
\begin{align}\label{eq: model problem}
    \begin{split}
        -\div(\vec{K}\nabla u^*) &= f  \quad \text{in } \Omega, \\
        u^* &= 0 \quad \text{on } \partial\Omega.
    \end{split}
\end{align}
Adaptive finite element methods (AFEMs) allow to achieve optimal convergence rates for such problems even if the solution exhibits singularities.
These methods usually require the solution of a sequence of linear systems arising from the discretization on successively refined meshes.
Solving these systems by direct methods is computationally too expensive to guarantee optimal error decay with respect to the overall computing time, making the use of iterative solvers into the adaptive algorithm essential.
Indeed, it has been shown that uniformly contractive iterative solvers are integral to achieve optimal complexity of AFEMs, i.e., optimal convergence rates with respect to the \emph{overall computational cost} and hence time; see, e.g.,~\cite{Ste07,GHPS21, BFMPS23}. 

For the model problem~\eqref{eq: model problem}, discretized by $H^1$-conforming finite elements, the resulting Galerkin matrix $\vec{A}_{\ell}$ is symmetric and positive definite (SPD), which naturally suggests the use of the conjugate gradient method (CG) as an algebraic solver.
However, the contraction factor of CG degrades with respect to the discretization parameters, i.e., the local mesh size $h$ and the polynomial degree $p$ (see, e.g.,~\cite{CG22}).
Therefore, appropriate $h$- and $p$-robust preconditioning is required to obtain optimal convergence rates.
Designing such preconditioners, which we will denote by $\vec{B}_{\ell}$, is nontrivial, as they must, first, approximate the inverse of the Galerkin matrix sufficiently well, i.e., $\vec{B}_{\ell} \approx \vec{A}_{\ell}^{-1}$, and, second, be cheap to apply, ideally with linear computational complexity .
Robust iterative solvers themselves, such as, e.g., the geometric multigrid method introduced in~\cite{IMPS24}, are natural candidates for this purpose, as they are designed to efficiently solve the Galerkin system.

The use of additive or multiplicative multilevel methods as preconditioners for CG in finite element discretizations has a long history.
Early results include~\cite{BP93} for multilevel preconditioners in the setting of quasi-uniform and locally refined grids under a \emph{nested refinement} assumption.
For multiplicative methods, this assumption is not merely technical and often necessitates additional preprocessing, such as the reconstruction of a virtual refinement hierarchy; see, e.g.,~\cite{BEK93,BY93,HZ09}. 
For graded bisection grids, where only a single bisection is performed per refinement step, uniform convergence with respect to $h$ was shown in~\cite{CNX08,CNX12}. Even in this setting, additional coarsening algorithms are typically required to meet the theoretical assumptions; see, e.g.,~\cite{CZ10}. In~\cite{WC06}, the idea of restricting levelwise smoothing to newly created vertices and their direct neighbors was introduced, yielding uniform convergence of multigrid methods for adaptive meshes generated by newest vertex bisection (NVB) in two dimensions. 
This $h$-robust approach was later extended to three dimensions in~\cite{WZ17}.
An $h$- and $p$-robust geometric multigrid method has been proposed and analyzed in~\cite{IMPS24}. 
Its \(h\)-robustness relies on the strengthened Cauchy--Schwarz inequality from~\cite{HWZ12} and the stable decomposition established in~\cite{WZ17}. 
Its \(p\)-robustness follows from the stable decomposition in~\cite{SMPZ08}, which was previously also used in~\cite{MPV20,MPV21} to construct a \(p\)-robust contractive multigrid method (however lacking $h$-robustness in case of non-uniform refinement).

Most of the aforementioned works on preconditioners focus on linear and symmetric multilevel methods in order to comply with the standard assumptions of the preconditioned conjugate gradient method (PCG). From a computational perspective, however, it can be advantageous to consider multigrid methods with post-smoothing only, which seem to preserve comparable contraction properties at reduced cost; see, e.g.,~\cite{DHM21,MPV21, IMPS24}. 
Moreover, the analysis of multiplicative methods often requires strong assumptions on the levelwise operators, such as contraction~\cite{CNX12}, or norm bounds strictly smaller than two~\cite{BPWX91}.
In contrast,~\cite{MPV21,IMPS24} employ optimal step-sizes on each level, simplifying the contraction proofs and yielding better numerical behavior, but resulting in non-linear and non-symmetric multigrid methods.
In numerical linear algebra, however, it is well known that PCG may fail if the preconditioner is not symmetric and positive definite.
That said, PCG has been adapted in the literature to also accommodate more general preconditioners; see, e.g.,~\cite{AV91,GY99,Not00,Bla02}. 
In particular, the generalized preconditioned conjugate gradient method (GPCG) introduced in~\cite{Bla02} also allows for non-symmetric and non-linear preconditioners.

This work establishes a direct connection between uniformly contractive iterative solvers arising in AFEM and their use as preconditioners within conjugate gradient methods.
In particular, we show that \emph{any} uniformly contractive solver with linear computational complexity induces an \emph{optimal preconditioner} for GPCG from~\cite{Bla02}, i.e., the resulting algebraic solver has a contraction factor independent of $h$ and $p$ and is of linear complexity. 
While this is certainly well-known in the numerical linear algebra community for uniform triangulations, we are not aware of a general presentation with a focus on adaptively generated triangulations.
Indeed, a central feature of the analysis in this work is that no additional preprocessing of the NVB-generated mesh hierarchy is required, allowing for seamless integration of the considered solvers into the adaptive finite-element algorithm.
Moreover, the presented framework unifies techniques from numerical linear algebra and AFEM analysis, thereby allowing contractive solvers to be used as preconditioners for GPCG irrespective of symmetry or linearity. 
To the best of our knowledge, this connection has not been addressed in the literature on optimal-complexity AFEM yet.

Based on the $h$- and $p$-robust stable decomposition from~\cite{IMPS24}, we derive various optimal preconditioners for GPCG and PCG. 
We first show that the geometric multigrid method from~\cite{IMPS24} directly induces a preconditioner that indeed fits in the general framework we develop in this work. In addition, we show that this multigrid method can be linearized and symmetrized, resulting in an SPD preconditioner suitable for standard PCG without loss of optimality. 
As additive and multiplicative Schwarz methods both rely on stable decompositions; see, e.g.,~\cite{BP93,CNX08,CNX12}, we use the same decomposition to also construct an optimal additive Schwarz preconditioner; see also~\cite{FFPE17} for a similar construction in the context of the $hp$-adaptive boundary element method.
Our numerical experiments show that the non-linear and non-symmetric multigrid preconditioner applied in GPCG outperforms the standalone multigrid solver as well as the additive Schwarz preconditioner used within PCG, both in terms of the experimental contraction factor and its application within the AFEM loop. 
The linear and symmetric multigrid preconditioner used in PCG generally exhibits better contraction factors as the iteration (and thus the Krylov space dimension) increase. 
However, its worse contraction in the initial iterations impacts the overall adaptive algorithm relying typically on few solver steps.
Moreover, the linear and symmetric version is computationally more expensive since it entails that pre-smoothing steps have to be added.

Finally, we note that all results in this work rely on the refined analysis from~\cite{Pau25}, which shows that the constants in the strengthened Cauchy--Schwarz inequality as well as in the \(h\)- and \(p\)-stable subspace decompositions depend only on the local behaviour of the diffusion coefficient.
In contrast, the constants in~\cite{IMPS24} and other works depend on the ratio of the maximal and minimal eigenvalue of the diffusion matrix, i.e., global diffusion contrast. \\

\textbf{Outline.}
In Section~\ref{sec: setting}, we introduce the problem setting. 
Section~\ref{sec: GPCG} is devoted to GPCG (Algorithm~\ref{algo: GPCG}) and its analysis (Proposition~\ref{theorem: general result}). 
In Section~\ref{sec: geometric multigrid}, we derive an optimal preconditioner for GPCG (Theorem~\ref{theorem: main result}) based on the geometric multigrid method (Algorithm~\ref{algo: geometric MG}) from~\cite{IMPS24}.
Section~\ref{sec: symmetric MG} presents a linearized and symmetrized variant (Algorithm~\ref{algo: symmetric MG}) and shows that this approach likewise yields an optimal preconditioner (Theorem~\ref{corollary: symmetric MG optimality}), where GPCG reduces to PCG.
Section~\ref{sec: additive Schwarz operator} introduces an additive Schwarz preconditioner based on the same space decomposition as for the proposed multigrid preconditioners and establishes its optimality (Theorem~\ref{theorem: AS condition number bound}).
We conclude the theoretical analysis in Section~\ref{sec: AFEM} with a brief discussion of the application of the proposed solvers within the AFEM framework (Algorithm~\ref{algo: AFEM}) to solve the symmetric model problem~\eqref{eq: model problem}. 
However, we already note here that the developments of the present work play a central role in the analysis of optimal-complexity AFEM for general second-order linear elliptic PDEs in the framework of the Lax--Milgram lemma, where the solution of the linear finite element system relies on optimal preconditioners for the principal part of the PDE; see~\cite{FHMP}.
Finally, Section~\ref{sec: numerics} presents numerical experiments comparing the developed algebraic solvers.

\section{Setting}\label{sec: setting}

\subsection{Mesh and space hierarchy}

Let $\TT_0$ be an initial conforming mesh of $\Omega$ into compact simplices $T \in \TT_0$, which is admissible in the sense of~\cite{Ste08}.
From now on, we consider a sequence $(\TT_{\ell})_{\ell \in \N_0}$ of successively refined triangulations obtained by newest vertex bisection; see, e.g.,~\cite{Tra97,Ste08} for $d \geq 2$ and~\cite{AFF13} for $d=1$. 
Thus, for all $\ell \geq 1$, it holds that $\TT_{\ell} = \texttt{refine}(\TT_{\ell-1},\MM_{\ell-1})$, where $\TT_{\ell}$ is the coarsest conforming triangulation obtained by NVB ensuring that all marked elements $\MM_{\ell-1}\subseteq \TT_{\ell-1}$ have been bisected. The generated sequence is uniformly $\gamma$-shape regular, i.e.,
\begin{equation}\label{eq: shape regularity}
	\sup_{\ell\in \N_0}\max_{T \in \TT_{\ell}} \frac{\diam(T)}{|T|^{1/d}} \leq \gamma < \infty \quad \text{and} \quad \sup_{\ell \in \N_0} \max_{T \in \TT_{\ell}} \max_{\substack{T' \in \TT_{\ell} \\ T \cap T' \neq \emptyset}} \frac{\diam(T)}{\diam(T')}\leq \gamma < \infty,
\end{equation}
where $\gamma$ depends only on $\TT_0$; see, e.g.,~\cite[Theorem 2.1]{Ste08} for $d\geq 2$ and~\cite{AFF13} for $d=1$. 
Furthermore, for every mesh $\TT_{\ell}$, we denote by $\VV_{\ell}$ the set of vertices. 
For $z \in \VV_{\ell}$, we define the $n$-patch $\TT^n_{\ell}(z)$ inductively via
\begin{equation*}
\TT_{\ell}(z) \coloneqq \TT^1_{\ell}(z) \coloneqq \{T \in \TT_{\ell}: z \in T\}, \quad \TT^{n+1}_{\ell}(z) \coloneqq \{T  \in \TT_{\ell} : T \cap \overline{\omega_{\ell}^n(z)} \neq \emptyset\},
\end{equation*}
where $\omega_{\ell}^n(z) \coloneqq \operatorname{int}(\bigcup_{T \in \TT^n_{\ell}(z)} T)$ denotes the corresponding $n$-patch subdomain.
With the element size $h_T \coloneqq |T|^{1/d}$ for all $T\in \TT_{\ell}$, the size of a patch subdomain is given by $h_{\ell,z} \coloneqq \max_{T \in \TT_{\ell}(z)} h_T$.
With $\omega_{\ell}(z) \coloneqq \omega_{\ell}^1(z)$, we finally define 
\begin{equation*}
	\VV_{0}^+ \coloneqq \VV_0 \quad \text{and} \quad \VV_{\ell}^+ \coloneqq \VV_{\ell}\backslash \VV_{\ell-1} \cup \{\VV_{\ell}\cap \VV_{\ell-1} : \omega_{\ell}(z) \neq \omega_{\ell-1}(z)\} \quad \text{for } \ell \in \N,
\end{equation*}
where $\VV_{\ell}^+$ consists of the new vertices of $\TT_{\ell}$ along with their neighboring vertices in $\VV_{\ell} \cap \VV_{\ell-1}$.

Let $q \geq 1$ and $T \in \TT_{\ell}$. 
Then, $\P^q(T)$ denotes the space of all polynomials on $T$ of degree at most $q$. 
For $\ell \in \N_0$, we define the finite-dimensional subspaces
\begin{equation*}
	\XX_{\ell}^q \coloneqq  \S_0^q(\TT_{\ell}) \coloneqq \{v_{\ell} \in H_0^1(\Omega) : v_{\ell}|_{T} \in \P^q(T) \text{ for all } T \in \TT_{\ell}\}  \subset \XX \coloneqq H_0^1(\Omega)
\end{equation*}
and observe that 
\begin{equation*}
	\XX_0^1 \subseteq \XX_1^1 \subseteq \cdots \subseteq \XX_{\ell-1}^1 \subseteq \XX_{\ell}^1 \subseteq \XX_{\ell}^p,
\end{equation*}
where $p \geq 1$ is a fixed polynomial degree.
Furthermore, the local spaces $\XX_{\ell,z}^q$ are given by
\begin{equation*}
	\XX_{\ell,z}^q \coloneqq \S_0^q(\TT_{\ell}(z)) \coloneqq \{v_{\ell} \in \XX_{\ell}^q : v_{\ell}|_{T} = 0 \text{ for all } T \in  \TT_{\ell} \backslash \TT_{\ell}(z)\}.
\end{equation*}
Denote by $\varphi^1_{\ell,z}$ the usual hat-function associated with the vertex $z \in \VV_{\ell}$, i.e., $\varphi_{\ell,z}^1|_{T} \in \P^1(T)$ for all $T \in \TT_{\ell}$ and $\varphi_{\ell,z}^1(z') = \delta_{zz'}$ for all $z' \in \VV_{\ell}$. 
The set $\{\varphi^1_{\ell,z} : z \in \VV_{\ell} \backslash \partial \Omega\}$ is a basis of $\XX_{\ell}^1$. Therefore, we define the spaces $\XX_{\ell}^+$ induced by the set of vertices $\VV_{\ell}^+$ as
\begin{equation*}
	\XX_{\ell}^+ \coloneqq \text{span}\{\varphi^1_{\ell,z} : z \in \VV_{\ell}^+ \backslash \partial \Omega\} \subseteq \XX_{\ell}^1.
\end{equation*}

\subsection{Weak formulation} 

For the analysis, we need some more assumptions on the model problem~\eqref{eq: model problem}.
Firstly, we will only consider $d \in \{1,2,3\}$.
For the diffusion coefficient $\vec{K}$, we require the stronger regularity $\vec{K}|_T \in [W^{1,\infty}(T)]^{d \times d}$ for all $T \in \TT_0$, where $\TT_0$ is the initial triangulation.
More precisely, this is needed to show the strengthened Cauchy--Schwarz inequality of Lemma~\ref{lemma: strengthened CS} and to define the residual error estimator~\eqref{eq: refinement indicators}.
For $x \in \Omega$, the expressions $\lambda_{\max}(\vec{K}(x))$ and $\lambda_{\min}(\vec{K}(x))$ denote the maximal and minimal eigenvalue of $\vec{K}(x) \in \R^{d \times d}$, respectively. 
For any measurable set $\omega \subseteq \Omega$, we  denote the $L^{2}(\omega)$-scalar product with $\langle \cdot , \cdot \rangle_{\omega}$. 
The weak formulation of~\eqref{eq: model problem} reads: 
Find $u^* \in \XX$ that solves
\begin{equation} \label{eq: weak formulation}
      \edual{u^{\star}}{v}_{\Omega} \coloneqq \langle \vec{K} \nabla u^{\star}, \nabla v \rangle_{\Omega} = \langle f, v \rangle_{\Omega} \eqqcolon F(v) \quad \text{for all } v \in \XX.
\end{equation}
In particular, the Riesz theorem yields existence and uniqueness of the weak solution $u^{\star} \in \XX$ to~\eqref{eq: weak formulation}.
From here on, we omit the index $\omega$ whenever $\omega=\Omega$. Furthermore, we define the induced norm $\enorm{v}^2 \coloneqq \edual{v}{v}$ and observe that
\begin{equation*}
	\inf_{y \in \omega} \lambda_{\min}(\vec{K}(y)) \, \|\nabla v \|^2_{\omega} \leq \enorm{v}_{\omega}^2 
	\leq \sup_{y \in \omega} \lambda_{\max}(\vec{K}(y)) \, \|\nabla v \|^2_{\omega} \quad \text{for all } v \in \XX \text{ and all } \omega \subseteq \Omega. 
\end{equation*}

For a given triangulation $\TT_{\ell}$ and polynomial degree $p \geq 1$, the Galerkin discretization of~\eqref{eq: model problem} reads: 
Find $u^{\star}_{\ell} \in \XX_{\ell}^p$ such that
\begin{equation}\label{eq: discrete problem}
	\edual{u^{\star}_{\ell}}{v_{\ell}} = F(v_{\ell}) \quad \text{for all } v_{\ell} \in \XX_{\ell}^p.
\end{equation}
Let $N_{{\ell}}^p \coloneqq \dim(\XX_{\ell}^p)$ and $\{\varphi^p_{{\ell},j}\}_{j =1}^{N_{\ell}^p}$ be a basis of $\XX_{\ell}^p$. Then, the discrete problem \eqref{eq: discrete problem} can equivalently be rewritten as
$\vec{A}_{\ell} \vec{x}^{\star}_{\ell} = \vec{b}_{\ell}$, with the symmetric and positive definite Galerkin matrix
\begin{equation}\label{eq: Galerkin matrix}
	(\vec{A}_{\ell})_{jk} \coloneqq \edual{\varphi^p_{\ell,j}}{\varphi^p_{\ell,k}} \quad \text{for } j,k = 1, \dots, N_{\ell}^p  
\end{equation}
and the right-hand side vector
\begin{equation}\label{eq: rhs vector}
	(\vec{b}_{\ell})_j \coloneqq F(\varphi^p_{\ell,j}) \quad \text{for } j = 1,\dots, N_{\ell}^p.
\end{equation}
The solution vector $\vec{x}_{\ell}^{\star} \in \R^{N_{\ell}^p}$ is the coefficient vector of the discrete solution $u_{\ell}^{\star} = \sum_{j=1}^{N_{\ell}^p} (\vec{x}_{\ell}^{\star})_j \varphi^p_{\ell,j}$ with respect to the fixed basis.
To describe this connection, we note that $\vec{x}_{\ell}^{\star} = \chi^p_{\ell}(u_{\ell}^{\star})$, where
\begin{equation}\label{eq: coefficient vector map}
	\chi^p_{\ell} \colon \XX_{\ell}^p \rightarrow \R^{N_{\ell}^p}, \quad \text{is defined via} \quad v_{\ell} = \sum_{j=1}^{N_{\ell}^p} \chi^p_{\ell}(v_{\ell})_j \varphi^p_{\ell,j}. 
\end{equation}

\section{Generalized preconditioned conjugate gradient method}\label{sec: GPCG}

In this section, we recall the generalized preconditioned conjugate gradient method from~\cite{Bla02}.
The contraction of the method relies on a discrete assumption of the preconditioner matrix; see~\cite[Theorem 3.4]{Bla02}.
We use this to show that any uniformly contractive algebraic solver with linear cost induces an optimal preconditioner for GPCG. Let $(\vec{x},\vec{y})_2 \coloneqq \vec{x}^T\vec{y}$ denote the Euclidean scalar product with corresponding norm $|\vec{x}|_{2} \coloneqq (\vec{x},\vec{x})_2^{1/2}$. 
For an SPD matrix $\vec{A} \in \R^{N \times N}$, we denote by $(\vec{x},\vec{y})_{\vec{A}} \coloneqq \vec{x}^T \vec{A}\vec{y}$ the induced scalar product with corresponding norm $|\vec{x}|_{\vec{A}}\coloneqq (\vec{x},\vec{x})_{\vec{A}}^{1/2}$.
Lastly, we denote by $\textnormal{cond}_2(\vec{M})$ and $\textnormal{cond}_{\vec{A}}(\vec{M})$ the condition number of a regular matrix $\vec{M} \in \R^{N \times N}$ with respect to the norms $|\cdot|_2$ and $|\cdot|_{\vec{A}}$, respectively.

\subsection{Symmetric preconditioners} 

In this section, we briefly recall the classical and preconditioned conjugate gradient method for solving linear systems $\vec{A}_{\ell}\vec{x}^{\star}_{\ell} = \vec{b}_{\ell}$ with a symmetric and positive definite matrix $\vec{A}_{\ell} \in \R^{N_{\ell}^p \times N_{\ell}^p}$. 
Let us begin with the CG algorithm; see, e.g.,~\cite{HS52}.

\begin{algorithm}[CG]\label{algo: CG}
	\textbf{Input:} SPD matrix $\vec{A}_{\ell} \in \R^{N_{\ell}^p \times N_{\ell}^p}$, right-hand side $\vec{b}_{\ell} \in \R^{N_{\ell}^p}$, initial guess $\vec{x}_{\ell}^0 \in \R^{N_{\ell}^p}$, and tolerance $\tau \geq 0$.  \\
	Set $\vec{p}_{\ell}^0 \coloneqq \vec{r}_{\ell}^0 = \vec{b}_{\ell} - \vec{A}_{\ell}\vec{x}_{\ell}^0$ and repeat for all $k=0,1,2, \ldots$ until $|\vec{r}_{\ell}^k|_2^2 <\tau$:
	\begin{enumerate}[label=\textnormal{(\roman*)}]
		\item \textbf{Update approximation:} $\vec{x}_{\ell}^{k+1} \coloneqq \vec{x}_{\ell}^k + \alpha_k \vec{p}_{\ell}^k$ with $\alpha_k = |\vec{r}_{\ell}^k|_2^2 /|\vec{p}_{\ell}^k|_{\vec{A}_{\ell}}^2$
		\item \textbf{Update residual:} $\vec{r}_{\ell}^{k+1} \coloneqq \vec{r}_{\ell}^k - \alpha_k \vec{A}_{\ell}\vec{p}_{\ell}^k$
		\item \textbf{Compute new search direction:} $\vec{p}_{\ell}^{k+1} \coloneqq \vec{r}_{\ell}^{k+1} + \beta_k \vec{p}_{\ell}^k$ with $\beta_k = |\vec{r}_{\ell}^{k+1}|_2^2 / |\vec{r}_{\ell}^k|_2^2$
	\end{enumerate}
	\textbf{Output:} Approximation $\vec{x}_{\ell}^{k}$ to the solution $\vec{x}_{\ell}^{\star}$ of $\vec{A}_{\ell}\vec{x}_{\ell}^{\star} = \vec{b}_{\ell}$.
\end{algorithm}

The following result gives a convergence estimate for the CG method. For more details and a proof, we refer to \cite[Theorem 11.3.3]{GV13}.

\begin{proposition}[Contraction of CG]\label{theorem: CG convergence}
	Let $\vec{A}_{\ell} \in \R^{N_{\ell}^p \times N_{\ell}^p}$ be a symmetric and positive definite matrix and let $\vec{x}_{\ell}^{\star} \in \R^{N_{\ell}^p}$ be the unique solution of the linear system $\vec{A}_{\ell}\vec{x}_{\ell}^{\star} = \vec{b}_{\ell}$. 
	Then, for any initial guess $\vec{x}_{\ell}^0 \in \R^{N_{\ell}^p}$, the sequence $\vec{x}_{\ell}^k  \in \R^{N_{\ell}^p}$ produced by Algorithm~\ref{algo: CG} guarantees contraction
	\begin{equation*}
		|\vec{x}_{\ell}^{\star} - \vec{x}_{\ell}^{k+1}|_{\vec{A}_{\ell}}\leq \Big(1-\frac{1}{\textnormal{cond}_2(\vec{A}_{\ell})}\Big)^{1/2}|\vec{x}_{\ell}^{\star} - \vec{x}_{\ell}^k|_{\vec{A}_{\ell}}. \qed
	\end{equation*}
\end{proposition}

For the Galerkin matrix $\vec{A}_{\ell}$, the condition number $\textnormal{cond}_2(\vec{A}_{\ell})$ depends on and grows with the local mesh size $h$ and the polynomial degree $p$. 
One way to achieve better contraction is to introduce an SPD preconditioner $\vec{B}_{\ell}$ such that the condition number of $\vec{B}_{\ell}^{1/2} \vec{A}_{\ell} \vec{B}_{\ell}^{1/2}$ is bounded independently of $h$ and $p$ and to apply CG to the modified system
\begin{equation}\label{eq: preconditioned system}
	\widetilde{\vec{A}}_{\ell} \widetilde{\vec{x}}^{\star}_{\ell}\coloneqq \vec{B}_{\ell}^{1/2} \vec{A}_{\ell} \vec{B}_{\ell}^{1/2} \widetilde{\vec{x}}_{\ell}^{\star} = \vec{B}_{\ell}^{1/2} \vec{b}_{\ell} \eqqcolon \widetilde{\vec{b}}_{\ell}.
\end{equation}
This is known as the preconditioned CG method (PCG).
Note that $\widetilde{\vec{A}}_{\ell}$ is SPD, so indeed CG can be applied.
For the initial guess $\vec{x}^{0}_{\ell} \in \R^{N_{\ell}^p}$, let $\widetilde{\vec{x}}^{0}_{\ell} \coloneqq \vec{B}_{\ell}^{-1/2} \vec{x}_{\ell}^{0}$ and let $\widetilde{\vec{x}}_{\ell}^k$, $\widetilde{\vec{r}}_{\ell}^k$, and $\widetilde{\vec{p}}_{\ell}^k$ denote the sequences generated by Algorithm~\ref{algo: CG} for the preconditioned system~\eqref{eq: preconditioned system}. 
Furthermore, let $\vec{x}_{\ell}^k \coloneqq \vec{B}_{\ell}^{1/2}\widetilde{\vec{x}}_{\ell}^k$, $\vec{r}_{\ell}^k \coloneqq \vec{B}_{\ell}^{-1/2}\widetilde{\vec{r}}_{\ell}^k$, and $\vec{p}_{\ell}^k \coloneqq \vec{B}_{\ell}^{1/2}\widetilde{\vec{p}}_{\ell}^k$. 
Then, there holds
\begin{align*}
	\vec{p}_{\ell}^0 &= \vec{B}_{\ell} \vec{r}_{\ell}^0, \hspace{-2.5cm}\\
	\vec{x}_{\ell}^{k+1} &= \vec{x}_{\ell}^k + \alpha_k \vec{p}_{\ell}^k, \hspace{-2.5cm} &&\text{where } \alpha_k = |\vec{r}_{\ell}^k|^2_{\vec{B}_{\ell}}/|\vec{p}_{\ell}^k|^{2}_{\vec{A}_{\ell}}, \\
	\vec{r}_{\ell}^{k+1} &= \vec{r}_{\ell}^k - \alpha_k \vec{A}_{\ell} \vec{p}_{\ell}^k, \hspace{-2.5cm} \\
	\vec{p}_{\ell}^{k+1} &= \vec{B}_{\ell}\vec{r}_{\ell}^{k+1} + \beta_k \vec{p}_{\ell}^k, \hspace{-2.5cm}&&\text{where } \beta_k =|\vec{r}_{\ell}^{k+1}|^2_{\vec{B}_{\ell}}/|\vec{r}_{\ell}^k|^2_{\vec{B}_{\ell}} 
\end{align*}
Moreover, for $\vec{y}_{\ell} \in \R^{N_{\ell}^p}$ and $\widetilde{\vec{y}}_{\ell} \coloneqq \vec{B}_{\ell}^{-1/2}\vec{y}_{\ell}$, we have the identity
\begin{align}\label{eq: norm equality sym matrices}
	\begin{split}
	|\vec{y}_{\ell}|_{\vec{A}_{\ell}}^2 &= (\vec{A}_{\ell}\vec{y}_{\ell},\vec{y}_{\ell})_2 = (\vec{A}_{\ell} \vec{B}_{\ell}^{1/2}\widetilde{\vec{y}}_{\ell}, \vec{B}_{\ell}^{1/2}\widetilde{\vec{y}}_{\ell})_2 \\
	&= (\vec{B}_{\ell}^{1/2}\vec{A}_{\ell}\vec{B}_{\ell}^{1/2}\widetilde{\vec{y}}_{\ell}, \widetilde{\vec{y}}_{\ell})_2 = |\widetilde{\vec{y}}_{\ell}|^2_{\vec{B}_{\ell}^{1/2}\vec{A}_{\ell}\vec{B}_{\ell}^{1/2}}.
	\end{split} 
\end{align}
With $q \coloneqq (1-1/\textnormal{cond}_2(\vec{B}_{\ell}^{1/2}\vec{A}_{\ell}\vec{B}_{\ell}^{1/2}))^{1/2}$, it hence follows from Proposition~\ref{theorem: CG convergence} that
\begin{equation*}
	|\vec{x}_{\ell}^{\star}-\vec{x}_{\ell}^k|_{\vec{A}_{\ell}} \overset{\eqref{eq: norm equality sym matrices}}{=} |\widetilde{\vec{x}}_{\ell}^{\star} - \widetilde{\vec{x}}_{\ell}^k|_{\vec{B}_{\ell}^{1/2}\vec{A}_{\ell}\vec{B}_{\ell}^{1/2}} 
	\leq q |\widetilde{\vec{x}}_{\ell}^{\star} - \widetilde{\vec{x}}_{\ell}^0|_{\vec{B}_{\ell}^{1/2}\vec{A}_{\ell}\vec{B}_{\ell}^{1/2}} \overset{\eqref{eq: norm equality sym matrices}}{=} q |\vec{x}_{\ell}^{\star}-\vec{x}_{\ell}^k|_{\vec{A}_{\ell}}.
\end{equation*}
Finally, \cite[Theorem C.1]{TW05} shows that
\begin{equation}\label{eq: condition numberfor PCG}
	\textnormal{cond}_{2}(\vec{B}_{\ell}^{1/2}\vec{A}_{\ell}\vec{B}_{\ell}^{1/2}) = \textnormal{cond}_{\vec{A}_{\ell}}(\vec{B}_{\ell}\vec{A}_{\ell}).
\end{equation}
Therefore, it suffices to show $h$- and $p$-independent boundedness of $\textnormal{cond}_{\vec{A}_{\ell}}(\vec{B}_{\ell}\vec{A}_{\ell})$.

\subsection{Non-linear and non-symmetric preconditioners}

Let $\vec{B}_{\ell}: \R^{N_{\ell}^p} \rightarrow \R^{N_{\ell}^p}$ denote a generally non-linear and non-symmetric preconditioner.
Then, the generalized preconditioned conjugate gradient method introduced in~\cite{Bla02} is given by the following algorithm.

\begin{algorithm}[GPCG]\label{algo: GPCG}
	\textbf{Input:} SPD matrix $\vec{A}_{\ell} \in \R^{N_{\ell}^p \times N_{\ell}^p}$, preconditioner $\vec{B}_{\ell}: \R^{N_{\ell}^p} \rightarrow \R^{N_{\ell}^p}$, right-hand side $\vec{b}_{\ell} \in \R^{N_{\ell}^p}$, initial guess $\vec{x}_{\ell}^0 \in \R^{N_{\ell}^p}$, and tolerance $\tau >0$. \\
	Set $\vec{r}_{\ell}^0 \coloneqq \vec{b}_{\ell}- \vec{A}_{\ell}\vec{x}_{\ell}^0$ and $\vec{p}_{\ell}^0 \coloneqq \vec{B}_{\ell}[\vec{r}_{\ell}^0]$ and repeat for all $k=0,1,2,\ldots$ until $|\vec{r}_{\ell}^k|^{2}_2 < \tau$: 
	\begin{enumerate}[label=\textnormal{(\roman*)}, ref = \textnormal{\roman*}]
		\item \textbf{Update approximation:} $\vec{x}_{\ell}^{k+1} \coloneqq \vec{x}_{\ell}^k + \alpha_k \vec{p}_{\ell}^k$ with $\alpha_k \coloneqq (\vec{B}_{\ell}[\vec{r}_{\ell}^k], \vec{r}_{\ell}^k)_2 /|\vec{p}_{\ell}^k|^2_{\vec{A}_{\ell}}$
		\item \textbf{Update residual:} $\vec{r}_{\ell}^{k+1} \coloneqq \vec{r}_{\ell}^k - \alpha_k \vec{A}_{\ell}\vec{p}_{\ell}^k$
		\item \label{item: GPCG step 3}\textbf{Compute new search direction:} $\vec{p}_{\ell}^{k+1} \coloneqq \vec{B}_{\ell}[\vec{r}_{\ell}^{k+1}] + \beta_k \vec{p}_{\ell}^k$ with 
		\begin{equation*}
			\beta_k \coloneqq \frac{(\vec{B}_{\ell}[\vec{r}_{\ell}^{k+1}], \vec{r}_{\ell}^{k+1})_2-(\vec{B}_{\ell}[\vec{r}_{\ell}^{k+1}], \vec{r}_{\ell}^{k})_2}{(\vec{B}_{\ell}[\vec{r}_{\ell}^{k}], \vec{r}_{\ell}^{k})_2} 
		\end{equation*}
	\end{enumerate}
	\textbf{Output:} Approximation $\vec{x}_{\ell}^{k}$ to the solution $\vec{x}_{\ell}^{\star}$ of $\vec{A}_{\ell}\vec{x}_{\ell}^{\star} = \vec{b}_{\ell}$.
\end{algorithm}

In~\cite{Bla02}, different assumptions on the preconditioner $\vec{B}_{\ell}$ are investigated. 
We will only consider the case where $\vec{B}_{\ell}$ is a good approximation of the SPD matrix $\vec{A}_{\ell}^{-1} \in \R^{N_{\ell}^p \times N_{\ell}^p}$, in the sense that there exists an $h$- and $p$-robust factor $q \in (0,1)$ such that
\begin{equation}\label{eq: preconditioner assumption}
	|\vec{B}_{\ell}[\vec{x}_{\ell}]-\vec{A}_{\ell}^{-1}\vec{x}_{\ell}|_{\vec{A}_{\ell}} \leq q \, |\vec{A}_{\ell}^{-1}\vec{x}_{\ell}|_{\vec{A}_{\ell}} \quad \text{for all } \vec{x}_{\ell} \in \R^{N_{\ell}^p}.
\end{equation}
In~\cite[Theorem 3.4]{Bla02}, contraction of GPCG under the assumption~\eqref{eq: preconditioner assumption} is shown. 
While~\eqref{eq: preconditioner assumption} is a discrete assumption, we use it to show the following central result.

\begin{proposition}[Contractive solvers are general preconditioners]\label{theorem: general result}
	Let $\SS_{\ell} : \XX_{\ell}^p \rightarrow \XX_{\ell}^p$ denote the error propagation operator of a given algebraic solver, i.e., $u_{\ell}^{k+1}-u_{\ell}^k = \SS_{\ell}(u_{\ell}^{\star}-u_{\ell}^k)$. 
	Assume the solver is contractive, i.e., there exists a constant $q \in (0,1)$ such that
	\begin{equation}\label{eq: contraction of operator}
		\enorm{(I-\SS_{\ell})u_{\ell}} \leq q \, \enorm{u_{\ell}} \quad \text{for all } u_{\ell} \in \XX_{\ell}^p.
	\end{equation}
	Let $\vec{B}_{\ell} : \R^{N_{\ell}^p} \rightarrow \R^{N_{\ell}^p}$ denote the representation of $\SS_{\ell}$ acting on coefficient vectors, meaning that $\chi^p_{\ell}(\SS_{\ell} u_{\ell}) = \vec{B}_{\ell}[\vec{A}_{\ell}\chi^p_{\ell}(u_{\ell})]$ for all $u_{\ell} \in \XX_{\ell}^p$, where $\vec{A}_{\ell}$ is the Galerkin matrix satisfying $\enorm{\cdot} = |\chi^p_{\ell}(\cdot)|_{\vec{A}_{\ell}}$.
	Then, the GPCG method with preconditioner $\vec{B}_{\ell}$ applied to the Galerkin system $\vec{A}_{\ell} \vec{x}_{\ell}^{\star} = \vec{b}_{\ell}$ yields an iterative solver with contraction factor $q$, i.e.,
	\begin{equation}\label{eq:preconditioned GPCG contraction}
		|\vec{x}_{\ell}^{\star}-\widetilde{\vec{x}}_{\ell}^{k+1}|_{\vec{A}_{\ell}} \leq q \, |\vec{x}_{\ell}^{\star}-\widetilde{\vec{x}}_{\ell}^{k}|_{\vec{A}_{\ell}} \quad \text{for all } k \in \N_0.
	\end{equation}
	For $u_{\ell}^{\star} = \sum_{j=1}^{N_{\ell}^p} (\vec{x}_{\ell}^{\star})_j \varphi^p_{\ell,j}$ and $\widetilde{u}_{\ell}^{k} = \sum_{j=1}^{N_{\ell}^p} (\widetilde{\vec{x}}_{\ell}^{k})_j \varphi^p_{\ell,j}$, this directly translates to 
	\begin{equation*}
		 \enorm{u_{\ell}^{\star}- \widetilde{u}_{\ell}^{k+1}} \leq q \, \enorm{u_{\ell}^{\star} - \widetilde{u}_{\ell}^k} \quad \text{for all } k \in \N_0. 
	\end{equation*}
\end{proposition}

\begin{proof}[\textbf{Proof}]
	Let $u_{\ell} \in \XX_{\ell}^p$ and $\vec{y}_{\ell} = \chi^p_{\ell}(u_{\ell})$.
	By the assumption $\chi^p_{\ell}(\SS_{\ell} u_{\ell}) = \vec{B}_{\ell}[\vec{A}_{\ell}\vec{y}_{\ell}]$, the definition of the energy norm, and~\eqref{eq: contraction of operator}, we have
	\begin{equation*}
		|\vec{y}_{\ell}-\vec{B}_{\ell}[\vec{A}_{\ell}\vec{y}_{\ell}]|_{\vec{A}_{\ell}} = \enorm{u_{\ell}-\SS_{\ell} u_{\ell}} \leq q \, \enorm{u_{\ell}} = q \, |\vec{y}_{\ell}|_{\vec{A}_{\ell}} \quad \text{for all } \vec{y}_{\ell} \in \R^{N_{\ell}^p}.
	\end{equation*}
	Choosing $\vec{y}_{\ell} = \vec{A}_{\ell}^{-1}\vec{x}_{\ell}$, we obtain assumption~\eqref{eq: preconditioner assumption}.
	Applying~\cite[Theorem 3.4]{Bla02} gives the desired contraction~\eqref{eq:preconditioned GPCG contraction}.
\end{proof}

\begin{remark}
The result of Proposition~\ref{theorem: general result} is entirely to be expected, since all quantities have been defined and denoted so as to fit into the framework of~\cite{Bla02}.
We emphasize, however, that in practice one must still verify that a given contractive solver can indeed be represented as a preconditioner, i.e., the property $\chi^p_{\ell}(\SS_{\ell} u_{\ell}) = \vec{B}_{\ell}[\vec{A}_{\ell}\chi^p_{\ell}(u_{\ell})]$ for all $u_{\ell} \in \XX_{\ell}^p$ has to be validated. 
For the geometric multigrid method from~\cite{IMPS24}, his is done in the next section.
\end{remark}

\section{Optimal non-linear and non-symmetric multigrid preconditioner}\label{sec: geometric multigrid}

In this section, we consider the $h$- and $p$-robustly contractive geometric multigrid method from~\cite{IMPS24} and rewrite the algorithm in matrix formulation with the purpose of verifying~\eqref{eq: preconditioner assumption} and using this solver as a (non-symmetric and non-linear) preconditioner for GPCG.

\subsection{$\textit{h}$- and $\textit{p}-$robust geometric multigrid method}\label{sec: hp multigrid}
For fixed $\ell$, we define the residual functional $R_{\ell}: \XX_{\ell}^p \rightarrow \R$ associated with the current approximation $u_{\ell}$ of $u_{\ell}^{\star}$ by $R_{\ell}(v_{\ell}) \coloneqq \edual{u_{\ell}^{\star} - u_{\ell}}{v_{\ell}}$. 
Then, one V-cycle of the geometric multigrid method is given by the following algorithm.

\begin{algorithm}[V-cycle of local multigrid method~\cite{IMPS24}]\label{algo: geometric MG}
	\textbf{Input:} Current approximation $u_{\ell} \in \XX_{\ell}^p$, triangulations $\{\TT_{\ell'}\}_{\ell'=0}^\ell$ and polynomial degree $p \geq 1$. \\
	Perform the following steps \textnormal{(i)--(iii)}:
	\begin{enumerate}[label=\textnormal{(\roman*)}, ref = \textnormal{\roman*}]
		\item \label{eq: MG step 1}\textbf{Lowest-order coarse solve:} Find $\rho_0 \in \XX_0^1$ such that 
		\begin{equation}\label{eq: coarse solve}
			\edual{\rho_0}{v_0} = R_{\ell}(v_0) \quad \text{for all } v_0 \in \XX_0^1. 
		\end{equation}
		Define $\sigma_0 \coloneqq \rho_0$ and $\lambda_0 \coloneqq 1$.

		\item \label{eq: MG step 2}\textbf{Local lowest-order correction:} For all intermediate levels $\ell' = 1, \dots, \ell-1$ and all $z \in \VV_{\ell'}^+$, compute $\rho_{\ell',z} \in \XX_{\ell',z}^1$ such that
		\begin{equation}\label{eq: local lowest order smoothing}
			\edual{\rho_{\ell',z}}{v_{\ell',z}} = R_{\ell}(v_{\ell',z}) - \edual{\sigma_{\ell'-1}}{v_{\ell',z}} \quad \text{for all } v_{\ell',z} \in \XX_{\ell',z}^1.
		\end{equation}
		Define $\rho_{\ell'} \coloneqq \sum_{z \in \VV_{\ell'}^+} \rho_{\ell',z}$, $\nu_{\ell'} \coloneqq \frac{R_{\ell}(\rho_{\ell'})- \edual{\sigma_{\ell'-1}}{\rho_{\ell'}}}{\enorm{\rho_{\ell'}}^2}$, and
		\begin{equation*}
			\sigma_{\ell'} \coloneqq  \sigma_{\ell'-1} + \lambda_{\ell'} \rho_{\ell'}, \quad \text{where} \quad
			\lambda_{\ell'} \coloneqq \begin{cases}
				\nu_{\ell'}  & \text{if }  \nu_{\ell'}\leq d+1, \\
				(d+1)^{-1} & \text{otherwise}.
			\end{cases}
		\end{equation*}

		\item \label{eq: MG step 3}\textbf{Local high-order correction:} For all $z \in \VV_{\ell}$, compute $\rho_{\ell,z} \in \XX_{\ell,z}^p$ such that
		\begin{equation}\label{eq: local high order problem}
			\edual{\rho_{\ell,z}}{v_{\ell,z}} = R_{\ell}(v_{\ell,z}) - \edual{\sigma_{\ell-1}}{v_{\ell,z}} \quad \text{for all } v_{\ell,z} \in \XX_{\ell,z}^p.
		\end{equation}
		Define $\rho_{\ell} \coloneqq \sum_{z \in \VV_{\ell}}\rho_{\ell,z}$ and
		\begin{equation*}
			\sigma_{\ell} \coloneqq  \sigma_{\ell-1} + \lambda_{\ell}\rho_{\ell}, \quad \text{where} \quad
			\lambda_{\ell} \coloneqq \frac{R_{\ell}(\rho_{\ell})- \edual{\sigma_{\ell-1}}{\rho_{\ell}}}{\enorm{\rho_{\ell}}^2}.
		\end{equation*} 
	\end{enumerate}
	\textbf{Output:} Improved approximation $\Phi_{\ell}(u_{\ell}) \coloneqq u_{\ell} + \sigma_{\ell} $.
\end{algorithm}

\begin{remark}
	In the case \(p=1\), it suffices to restrict step~\eqref{eq: MG step 3} of Algorithm~\ref{algo: geometric MG} to \(\VV_{\ell}^+\), i.e., the local problems~\eqref{eq: local high order problem} need only be solved for \(z \in \VV_{\ell}^+\).
	All comments and results then hold accordingly.
\end{remark}

Using the Galerkin matrix $\vec{A}_{\ell}$ from~\eqref{eq: Galerkin matrix} and the vector $\vec{b}_{\ell}$ from~\eqref{eq: rhs vector}, we want to rewrite Algorithm~\ref{algo: geometric MG} in terms of linear algebra.
To this end, we introduce the following notation:
For $\ell' \in \{0,\dots, \ell\}$, consider, only for an analytical perspective, the embedding $\II^+_{\ell'}: \XX_{\ell'}^+ \rightarrow \XX_{\ell}^p$, i.e., the formal identity, with matrix representation $\vec{I}^+_{\ell'} \in \R^{N_{\ell}^p \times N_{\ell'}^+}$, where $N_{\ell'}^+ \coloneqq \dim(\XX_{\ell'}^+)$.  
Similarly, the embeddings $\II_{\ell',z}^p : \XX_{\ell',z}^p \rightarrow \XX_{\ell}^p$ also have matrix representations $\vec{I}_{\ell',z}^p \in \R^{N_{\ell}^p \times N_{\ell',z}^p}$ with $N_{\ell',z}^p \coloneqq \dim(\XX_{\ell',z}^p)$. 
Denote by $\vec{A}_{\ell'}$, $\vec{A}_{\ell'}^+$, and $\vec{A}_{\ell',z}^p$ the Galerkin matrices with respect to $\XX_{\ell'}^1$, $\XX_{\ell'}^+$, and $\XX_{\ell',z}^p$, respectively. 
We define the diagonal matrices $\vec{D}_{\ell'}^+ \in \R^{N_{\ell'}^+ \times N_{\ell'}^+}$ via $(\vec{D}^+_{\ell'})_{jk} \coloneqq \delta_{jk} (\vec{A}^+_{\ell'})_{jj}$.
Furthermore, consider the levelwise smoothers 
\begin{align}\label{eq: levelwise matrices}
	\begin{split}
	\vec{S}_0 &\coloneqq \vec{I}^+_0 \vec{A}_0^{-1} (\vec{I}^+_0)^T\vec{A}_{\ell}, \quad \vec{S}_{\ell'} \coloneqq   \, \vec{I}^+_{\ell'}(\vec{D}_{\ell'}^{+})^{-1} (\vec{I}^+_{\ell'})^T \vec{A}_{\ell} \quad \text{for } \ell' = 1, \dots, \ell-1, \quad \text{and} \\
	\vec{S}_{\ell} &\coloneqq  \sum_{z \in \VV_{\ell}} \vec{I}_{\ell,z}^p (\vec{A}_{\ell,z}^p)^{-1} (\vec{I}_{\ell,z}^p)^T \vec{A}_{\ell}.
	\end{split}
\end{align}
Lastly, we note that matrix products, and analogously products of operators, are generally non-commutative.
In particular, for matrices or operators $C_0,\dots,C_{\ell}$, we define
\begin{equation*}
	\prod_{\ell'=0}^{\ell} C_{\ell'} \coloneqq  C_0 C_1 \cdots C_{\ell} \quad \text{as well as} \quad \prod_{\ell' = \ell}^0 C_{\ell'} \coloneqq C_\ell C_{\ell-1} \cdots C_0.
\end{equation*}

In~\cite{IMPS24}, it is already remarked that the solver from Algorithm~\ref{algo: geometric MG} is of linear complexity per step. 
However, let us provide some more details. 

\begin{remark}[Computational complexity of Algorithm~\ref{algo: geometric MG}]\label{rem: complexity of geometric MG}
	The computational cost on the initial mesh depends only on $\# \TT_0$. 
	The matrices for the local high-order problems~\eqref{eq: local high order problem} have dimension $\OO(p^d)$, where the notationally hidden constant depends only on $\gamma$-shape regularity.
	Since we solve such a system on every patch, the computational complexity on the finest level is of order $\OO(p^{3d}\# \TT_{\ell})$.
	As $\dim(\XX_{\ell',z}^1)=1$, each of local lowest-order problems~\eqref{eq: local lowest order smoothing} can be solved in $\OO(1)$ operations.
	Hence, the computation of $\rho_{\ell'}$ is of order $\OO(\# \VV_{\ell'}^+)$ for $\ell' = 1, \dots, \ell-1$.
	Moreover, let us discuss the calculation of the step-size $\lambda_{\ell'}$ when it is bounded by $(d+1)$. 
	By definition of the local problems~\eqref{eq: local lowest order smoothing}, we have
	\begin{equation*}
		R_{\ell}(\rho_{\ell'}) - \edual{\sigma_{\ell'-1}}{v_{\ell',z}} = \sum_{z \in \VV_{\ell'}^+} \enorm{\rho_{\ell',z}}^2
	\end{equation*}
	and the resulting sum can be computed in $\OO(\# \VV_{\ell'}^+)$ operations.
	Let $\vec{s}_{\ell'} = \chi_{\ell'}^1[\rho_{\ell'}]$ denote the vector corresponding to $\rho_{\ell'}$. 
	Since the Galerkin matrix $\vec{A}_{\ell'}$ is sparse and we are solely interested in the value $\enorm{\rho_{\ell'}}^2 = \vec{s}_{\ell'}^T \vec{A}_{\ell'}\vec{s}_{\ell'}$ and $\rho_{\ell'} \in \XX_{\ell'}^+$, we only need to compute the entries of $\vec{A}_{\ell'} \vec{s}_{\ell'}$ corresponding to the vertices in $\VV_{\ell'}^+$.
	This can be done in $\OO(\# \VV_{\ell'}^+)$ operations.
	Thus, the overall computational complexity on an intermediate level $\ell'$ is of order $\OO(\# \VV_{\ell'}^+)$.
	By the definition of $\VV_{\ell'}^+$, we have 
	\begin{equation*}
		\sum_{\ell'=1}^{\ell-1} \# \VV_{\ell'}^+ \lesssim \sum_{\ell'=1}^{\ell-1} \# (\VV_{\ell'}\backslash \VV_{\ell'-1}) = \sum_{\ell'=1}^{\ell-1} (\# \VV_{\ell'}-\# \VV_{\ell'-1}) = \# \VV_{\ell-1} - \# \VV_{0} \leq \# \VV_{\ell} \simeq \# \TT_{\ell}.
	\end{equation*}
	Therefore, the overall computational complexity of Algorithm~\ref{algo: geometric MG} is of order $\OO(\# \TT_{\ell})$ and the notationally hidden constant depends only on the initial mesh $\TT_0$, the polynomial degree $p$, the dimension $d$, and $\gamma$-shape regularity.
	In particular, the overall complexity does not depend on the number $\ell + 1$ of triangulations $\TT_{0}, \dots, \TT_{\ell}$.

Note that the matrices $\vec{I}_{\ell'}^+$ and $\vec{I}_{\ell,z}^p$ above are introduced solely for the purpose of the analysis and are not computed in the actual implementation.
In practice, Algorithm~\ref{algo: geometric MG} only employs the prolongation and restriction matrices between consecutive levels. 
These can be assembled in $\mathcal{O}(\# \VV_{\ell'}^+)$ operations on the intermediate levels and in $\mathcal{O}(\# \TT_{\ell})$ operations on the finest level, thus preserving the overall linear complexity of the method.
This recursive prolongation and restriction between consecutive levels is standard practice in multigrid methods; see, e.g.,~\cite{XQ94}. 
\end{remark}

Let us denote the algebraic residual of the current iterate $u_{\ell} \in \XX_{\ell}$ by $\vec{r}_{\ell} \coloneqq \vec{b}_{\ell} - \vec{A}_{\ell} \vec{x}_{\ell} \in \R^{N_{\ell}^p}$, where $\vec{x}_{\ell} = \chi^p_{\ell}(u_{\ell}) \in \R^{N_{\ell}^p}$.
With the matrices from~\eqref{eq: levelwise matrices}, we obtain 

\begin{enumerate}[label=\textnormal{(\roman*)}]
	\item \textbf{Lowest-order coarse solve:} 
	Let $\widetilde{\vec{s}}_0 = \chi_0^1(\rho_0) \in \R^{N^1_0}$ denote the coefficient vector of the solution $\rho_0 \in \XX_0^1$ to the coarse problem~\eqref{eq: coarse solve}. 
    Since $\VV_0^+ = \VV_0$ and hence $\XX_0^+ = \XX_0$, problem~\eqref{eq: coarse solve} can equivalently be reformulated by the linear system $\vec{A}_0\widetilde{\vec{s}}_0 = (\vec{I}^+_0)^T \vec{r}_{\ell}$.
	Due to the nestedness of the finite element spaces, it holds $\rho_0 \in \XX_{\ell}^p$ and thus $\vec{s}_0 = \vec{I}^+_0 \widetilde{\vec{s}}_0$ for $\vec{s}_0 = \chi^p_{\ell}(\rho_0) \in \R^{N_{\ell}^p}$. Hence, we have
	\begin{equation*}
		\vec{s}_0 = \vec{I}_0^+ \widetilde{\vec{s}}_0 = \vec{I}^+_0 (\vec{A}_0)^{-1} (\vec{I}^+_0)^T \vec{r}_{\ell} \eqqcolon \vec{B}_0 [\vec{r}_{\ell}]. 
	\end{equation*}

	\item \textbf{Local lowest-order correction:} 
	Combining the local contributions $\rho_{1,z}$ on the first level, we consider $\vec{s}_1 = \chi^p_{\ell}(\rho_1)$. 
	From~\eqref{eq: local lowest order smoothing}, it follows that
	\begin{align*}
		\vec{s}_1 &=  \vec{I}^+_1(\vec{D}_1^+)^{-1}(\vec{I}^+_1)^T(\vec{r}_{\ell}-\lambda_0\vec{A}_{\ell}\vec{s}_0)  \\
		&= \vec{I}^+_1(\vec{D}_1^+)^{-1}(\vec{I}^+_1)^T(\vec{I}-\lambda_0\vec{A}_{\ell}\vec{I}^+_0(\vec{A}_0)^{-1}(\vec{I}^+_0)^T)\vec{r}_{\ell} \eqqcolon \vec{B}_1 [\vec{r}_{\ell}].
	\end{align*}
	For $\vec{s}_2 = \chi^p_{\ell}(\rho_2) \in \R^{N_\ell^p}$, we can therefore show
	\begin{align*}
		&\vec{s}_2 =  \vec{I}^+_2(\vec{D}_2^+)^{-1}(\vec{I}^+_2)^T (\vec{r}_{\ell} - \vec{A}_{\ell}(\lambda_1\vec{s}_1+\lambda_0\vec{s}_0))  \\
		&=  \vec{I}^+_2(\vec{D}_2^+)^{-1}(\vec{I}^+_2)^T \Big(\vec{r}_{\ell}- \lambda_1\vec{A}_{\ell} \vec{I}^+_1(\vec{D}_1^+)^{-1}(\vec{I}^+_1)^T(\vec{r}_{\ell}-\lambda_0\vec{A}_{\ell}\vec{B}_0[\vec{r}_{\ell}]) - \lambda_0\vec{A}_{\ell} \vec{B}_0[\vec{r}_{\ell}] \Big) \\
		&=  \vec{I}^+_2(\vec{D}_2^+)^{-1}(\vec{I}^+_2)^T \big(\vec{I}-  \lambda_1\vec{A}_{\ell} \vec{I}^+_1(\vec{D}_1^+)^{-1}(\vec{I}_1^+)^T\big)\big(\vec{I}-\lambda_0\vec{A}_{\ell} \vec{I}^+_0(\vec{A}_0)^{-1}(\vec{I}^+_0)^T\big) \vec{r}_{\ell} \eqqcolon \vec{B}_2[\vec{r}_{\ell}].
	\end{align*}
	For easier notation, we set $\vec{D}_0^+ \coloneqq \vec{A}_0$. 
	By induction on $\ell'$, we obtain the update
	\begin{equation}\label{eq: levelwise preconditioner}
		\vec{s}_{\ell'} = \vec{I}^+_{\ell'} (\vec{D}_{\ell'}^+)^{-1} (\vec{I}^+_{\ell'})^T \prod_{i=\ell'-1}^{0} (\vec{I}- \lambda_{i} \vec{A}_{\ell}\vec{I}^+_{i}(\vec{D}_{i}^{+})^{-1} (\vec{I}^+_{i})^T)\vec{r}_{\ell} \eqqcolon \vec{B}_{\ell'}[\vec{r}_{\ell}],
	\end{equation}
	where $\vec{s}_{\ell'} = \chi^p_{\ell}(\rho_{\ell'}) \in \R^{N_{\ell}^p}$ and $\ell' = 1, \dots, \ell-1$.

	\item \textbf{Local high-order correction:} On the finest level, we have
	\begin{equation*}
		\vec{s}_{\ell} = \sum_{z \in \VV_{\ell}} \vec{I}_{\ell,z}^p (\vec{A}_{\ell,z}^p)^{-1} (\vec{I}_{\ell,z}^p)^T \prod_{\ell'= \ell-1}^{0} (\vec{I}- \lambda_{\ell'}\vec{A}_{\ell} \vec{I}^+_{\ell'} (\vec{D}_{\ell'}^{+})^{-1} (\vec{I}^+_{\ell'})^T) \vec{r}_{\ell}
		\eqqcolon \vec{B}_{\ell} [\vec{r}_{\ell}].
	\end{equation*}
\end{enumerate}

Thus, we obtain the update
\begin{equation}\label{eq: MG update}
	\chi^p_{\ell}(\Phi_{\ell}(u_{\ell})) = \vec{x}_{\ell} + \sum_{\ell'=0}^\ell \lambda_{\ell'}\vec{s}_{\ell'} = \vec{x}_{\ell} + \sum_{\ell'=0}^\ell \lambda_{\ell'} \vec{B}_{\ell'}[\vec{r}_{\ell}].
\end{equation}

\begin{remark}[Nonlinearity of the preconditioner]
	We note that the step-sizes $\lambda_{\ell'}$ for $\ell' \in \{1,\dots,\ell\}$ are also functions of the current residual $\vec{r}_{\ell}$, i.e., $\lambda_{\ell'} = \lambda_{\ell'}[\vec{r}_{\ell}]$. 
	More precisely, for $\ell' < \ell$ the step-sizes are defined as
	\begin{equation*}
		\lambda_{\ell'}[\vec{r}_{\ell}] \coloneqq \begin{cases}
			\nu_{\ell'}[\vec{r}_{\ell}] & \text{if } \nu_{\ell'}[\vec{r}_{\ell}] \leq d+1 \\
			(d+1)^{-1} & \text{otherwise},
		\end{cases}
	\end{equation*}
	where
	\begin{equation*}
		\nu_{\ell'}[\vec{r}_{\ell}] \coloneqq \frac{(\vec{r}_{\ell},\vec{B}_{\ell'}[\vec{r}_{\ell}])_2 - \big(\sum_{i=0}^{\ell'-1}\lambda_i[\vec{r}_{\ell}]\vec{B}_{i}[\vec{r}_{\ell}],\vec{B}_{\ell'}[\vec{r}_{\ell}]\big)_{\vec{A}_{\ell}}}{|\vec{B}_{\ell'}[\vec{r}_{\ell}]|^2_{\vec{A}_{\ell}}}.
	\end{equation*}	
	In the case $\ell' = \ell$, we have $\lambda_{\ell}[\vec{r}_{\ell}] = \nu_{\ell}[\vec{r}_{\ell}]$.
	From this, it is clear that the operators $\vec{B}_{\ell'} : \R^{N_{\ell}^p} \rightarrow \R^{N_{\ell}^p}$  are indeed non-linear.
	For easier notation, we will omit the dependence on $\vec{r}_{\ell}$ in the following when it is clear from the context.
	Moreover, we set $\lambda_0[\vec{r}_{\ell}] \coloneqq 1$.
\end{remark}

Consider the orthogonal projections $\PP^q_{\ell'}: \XX_{\ell}^p \to \XX_{\ell'}^q$ and  $\PP_{\ell',z}^q: \XX_{\ell}^p \to \XX_{\ell',z}^q$ which, for $q \in \{1,p\}$ and each $v \in \XX_{\ell}^p$, are given by
\begin{align}
	\edual{\PP^q_{\ell'} v}{w_{\ell'}} &= \edual{v}{w_{\ell'}}  \hspace{-1.5cm}&& \text{for all } w_{\ell'} \in \XX_{\ell'}^q, \label{eq: global projections}\\
	\edual{\PP_{\ell',z}^q v}{w_{\ell',z}} &= \edual{v}{w_{\ell',z}}  \hspace{-1.5cm}&& \text{for all } w_{\ell',z} \in \XX_{\ell',z}^q. \label{eq: local projections}
\end{align}
With these, we can define levelwise operators $\SS_0 : \XX_{\ell}^p \rightarrow \XX_0^1$, $\SS_{\ell'} : \XX_{\ell}^p \rightarrow \XX_{\ell'}^+$ for $\ell' = 1, \dots, \ell-1$, and $\SS_{\ell} : \XX_{\ell}^p \rightarrow \XX_{\ell}^p$ via
\begin{equation}\label{eq: levelwise operators}
	\SS_0 \coloneqq \PP_0^1, \quad \SS_{\ell'} \coloneqq  \sum_{z \in \VV_{\ell'}^+} \PP^1_{\ell',z} \quad \text{for } \ell' = 1, \dots, \ell-1, \quad \text{and} \quad \SS_{\ell} \coloneqq  \sum_{z \in \VV_{\ell}}\PP_{\ell,z}^p.
\end{equation}
The following lemma shows the connection between the levelwise operators $\SS_{\ell'}$ from~\eqref{eq: levelwise operators} and the matrices $\vec{S}_{\ell'}$ from~\eqref{eq: levelwise matrices}.

\begin{lemma}[Levelwise smoothers functional/matrix representation]\label{lemma: levelwise operator and Matrix}
	Let $\ell' \in \{0, \dots, \ell\}$.
	Then, it holds that
	\begin{equation}\label{eq: levelwise operator and Matrix}
		\edual{\SS_{\ell'} v_{\ell}}{w_{\ell}} = ( \vec{S}_{\ell'} \vec{x}_{\ell},\vec{y}_{\ell})_{\vec{A}_{\ell}} \quad \text{for all } v_{\ell},w_{\ell} \in \XX_{\ell}^p \text{ and } \vec{x}_{\ell} = \chi^p_{\ell}(v_{\ell}), \, \vec{y}_{\ell} = \chi^p_{\ell}(w_{\ell}).
	\end{equation}
\end{lemma}

\begin{proof}[\textbf{Proof}]
	Let $v_{\ell},w_{\ell} \in \XX_{\ell}^p$, $\vec{y}_{\ell} = \chi^p_{\ell}(w_{\ell}), \vec{x}_{\ell} = \chi^p_{\ell}(v_{\ell}), \vec{x}_0 = \chi^p_{\ell}(\PP^1_0v_{\ell})$, and $\widetilde{\vec{x}}_0 = \chi_0^1(\PP_0^1v_{\ell}) \in \R^{N_0^1}$. 
	Then, the definition~\eqref{eq: global projections} of $\PP^1_0$ implies that $\vec{A}_0 \widetilde{\vec{x}}_0 = (\vec{I}^+_0)^T \vec{A}_{\ell} \vec{x}_{\ell}$. Using $\SS_0 = \PP_0^1$, we thus get
	\begin{equation}\label{eq: coarse solve matrix}
		\vec{x}_0 = \vec{I}_0^+ \widetilde{\vec{x}}_0 = \vec{I}_0^+ \vec{A}_0^{-1} (\vec{I}^+_0)^T \vec{A}_{\ell} \vec{x}_{\ell} \eqreff{eq: levelwise matrices}= \vec{S}_0 \vec{x}_{\ell}.
	\end{equation}
	Therefore, we have
	\begin{equation*}
		\edual{\SS_0 v_{\ell}}{w_{\ell}} = (\vec{x}_0,\vec{y}_{\ell})_{\vec{A}_{\ell}} \overset{\eqref{eq: coarse solve matrix}}{=} (\vec{S}_0 \vec{x}_{\ell},\vec{y}_{\ell})_{\vec{A}_{\ell}}.
	\end{equation*}
	For $\ell'=1, \dots, \ell-1$, we similarly find
	\begin{equation*}
		\vec{x}_{\ell'} = \vec{I}_{\ell'}^+ (\vec{D}_{\ell'}^+)^{-1} (\vec{I}_{\ell'}^+)^T \vec{A}_{\ell} \vec{x}_{\ell} = \vec{S}_{\ell'} \vec{x}_{\ell}
	\end{equation*} 
	and 
	\begin{equation*}
		\edual{\SS_{\ell'} v_{\ell}}{w_{\ell}} = (\vec{x}_{\ell'},\vec{y}_{\ell})_{\vec{A}_{\ell'}} = (\vec{S}_{\ell'} \vec{x}_{\ell},\vec{y}_{\ell})_{\vec{A}_{\ell}}.
	\end{equation*}
	Finally, for the finest level $\ell$, we obtain~\eqref{eq: levelwise operator and Matrix} directly from the definiton~\eqref{eq: local projections} of $\PP^p_{\ell,z}$ and summation over $z \in \VV_{\ell}$.
	This concludes the proof.
\end{proof}

\subsection{Induced multigrid preconditioner and corresponding main result}\label{sec: induced MG preconditioner}

With the notation from the last section, we define the multigrid preconditioner
\begin{equation}\label{eq: MG preconditioner}
	\vec{B}_{\ell}^{\MG} \coloneqq \sum_{\ell'=0}^\ell \lambda_{\ell'} \vec{B}_{\ell'}: \R^{N_{\ell}^p} \rightarrow \R^{N_{\ell}^p}.
\end{equation}

The subsequent theorem shows that GPCG with preconditioner $\vec{B}_{\ell}^{\MG}$ contracts by an $h$- and $p$-independent factor, where the proof is postponed to Section~\ref{subsec: prove GPCG MG}.
Together with Remark~\ref{rem: complexity of geometric MG}, it follows that the preconditioner $\vec{B}_{\ell}^{\MG}$ is indeed optimal.

\begin{theorem}[GPCG with non-linear and non-symmetric MG preconditioner]\label{theorem: main result}
	Let $\vec{x}_{\ell}^k$ be the GPCG iterates generated by Algorithm~\ref{algo: GPCG} applied to the Galerkin system $\vec{A}_{\ell} \vec{x}_{\ell}^{\star} = \vec{b}_{\ell}$  employing the non-linear and non-symmetric multigrid preconditioner $\vec{B}_{\ell}^{\MG}$ defined in~\eqref{eq: MG preconditioner}.
	Then, there holds
	\begin{equation}\label{eq: MG preconditioned GPCG contraction}
		|\vec{x}_{\ell}^{\star}-\vec{x}_{\ell}^{k+1}|_{\vec{A}_{\ell}} \leq q \, |\vec{x}_{\ell}^{\star}-\vec{x}_{\ell}^k|_{\vec{A}_{\ell}} \quad \text{for all } k \in \N_0,
	\end{equation}
	where $q$ is indepedent of $h$, $p$, and only depends locally on the diffusion contrast $\vec{K}$.
	For $u_{\ell}^{\star} = \sum_{j=1}^{N_{\ell}^p} (\vec{x}_{\ell}^{\star})_j \varphi^p_{\ell,j}$ and $u_{\ell}^{k} = \sum_{j=1}^{N_{\ell}^p} (\vec{x}_{\ell}^{k})_j \varphi^p_{\ell,j}$, this directly translates to 
	\begin{equation*}
		 \enorm{u_{\ell}^{\star}- u_{\ell}^{k+1}} \leq q \, \enorm{u_{\ell}^{\star} - u_{\ell}^k} \quad \text{for all } k \in \N_0. 
	\end{equation*}
\end{theorem}

The following auxiliary lemma simplifies the application of the multigrid preconditioner~\eqref{eq: MG preconditioner}.

\begin{lemma}\label{lemma: unnötig}
	Let $\{\vec{M}_{\ell'}\}_{\ell'=0}^{\ell}$ be a family of matrices. Then, it holds that
	\begin{equation}\label{eq: unnötig}
		\sum_{\ell'=0}^{\ell} \vec{M}_{\ell'} \prod_{i=\ell'-1}^{0} (\vec{I}-\vec{M}_{i}) = \vec{I} - \prod_{\ell'= \ell}^{0} (\vec{I}-\vec{M}_{\ell'}).
	\end{equation}
\end{lemma}

\begin{proof}
	The proof is done by induction on $\ell$.
	For $\ell = 0$, we have
	\begin{equation*}
		\vec{M}_0 = \vec{I} - (\vec{I}-\vec{M}_0).
 	\end{equation*}
	Using the induction hypothesis for $\ell-1$, we obtain
	\begin{align*}
		\vec{I} - \prod_{\ell'= \ell}^{0}(\vec{I}-\vec{M}_{\ell'}) &= \vec{I} - (\vec{I}-\vec{M}_{\ell})\prod_{\ell'=\ell-1}^{0}(\vec{I}-\vec{M}_{\ell'}) = \vec{I} - \prod_{\ell'= \ell-1}^{0} (\vec{I}-\vec{M}_{\ell'}) + \vec{M}_{\ell} \prod_{i=\ell-1}^{0} (\vec{I}-\vec{M}_{i})  \\
		&= \sum_{\ell'=0}^{\ell-1} \vec{M}_{\ell'}\prod_{i=\ell'-1}^{0} (\vec{I}-\vec{M}_{i}) + \vec{M}_{\ell} \prod_{i= \ell-1}^{0}(\vec{I}-\vec{M}_{i})
		= \sum_{\ell'=0}^{\ell} \vec{M}_{\ell'} \prod_{i=\ell'-1}^{0} (\vec{I}-\vec{M}_{i}).
	\end{align*}
	This concludes the proof.
\end{proof}

The application of the multigrid preconditioner $\vec{B}_{\ell}^{\MG}$ to $\vec{A}_{\ell}\vec{x}_{\ell}$ for $\vec{x}_{\ell} \in \R^{N_{\ell}^p}$ can hence be simplified using the levelwise matrices $\vec{S}_{\ell'}$ defined in~\eqref{eq: levelwise matrices}.
First, we note that
\begin{align}\label{eq: levelwise matrices B}
	\begin{split}
	\vec{B}_{\ell'}[\vec{A}_{\ell}\vec{x}_{\ell}] \, &\eqreff*{eq: levelwise preconditioner}{=} \, \vec{I}^+_{\ell'} (\vec{D}_{\ell'}^+)^{-1} (\vec{I}^+_{\ell'})^T \prod_{i = \ell' -1}^{0} (\vec{I}- \lambda_{i}[\vec{A}_{\ell}\vec{x}_{\ell}] \vec{A}_{\ell}\vec{I}^+_{i}(\vec{D}_{i}^{+})^{-1} (\vec{I}^+_{i})^T)\vec{A}_{\ell} \vec{x}_{\ell} \\
	&= \vec{I}^+_{\ell'} (\vec{D}_{\ell'}^+)^{-1} (\vec{I}^+_{\ell'})^T \vec{A}_{\ell}\prod_{i = \ell'-1}^{0} (\vec{I}- \lambda_{i}[\vec{A}_{\ell}\vec{x}_{\ell}] \vec{I}^+_{i}(\vec{D}_{i}^{+})^{-1} (\vec{I}^+_{i})^T \vec{A}_{\ell}) \vec{x}_{\ell} \\
	&= \vec{S}_{\ell'} \prod_{i=\ell'-1}^{0} (\vec{I}- \lambda_{i}[\vec{A}_{\ell}\vec{x}_{\ell}] \vec{S}_{i})\vec{x}_{\ell}.
	\end{split}
\end{align}
Applying Lemma~\ref{lemma: unnötig} to $\vec{M}_{\ell'} = \lambda_{\ell'}[\vec{A}_{\ell}\vec{x}_{\ell}]\vec{S}_{\ell'}$, we thus observe that 
\begin{align*}\label{eq: MG preconditioner application}
	\begin{split}
	\vec{B}_{\ell}^{\MG}[\vec{A}_{\ell}\vec{x}_{\ell}] \overset{\eqref{eq: MG preconditioner}}{=} \sum_{\ell'=0}^\ell \lambda_{\ell'}[\vec{A}_{\ell}\vec{x}_{\ell}] \vec{B}_{\ell'} [\vec{A}_{\ell}\vec{x}_{\ell}] \, &\overset{\mathclap{\eqref{eq: levelwise matrices B}}}{=} \, \vec{S}_0 \vec{x}_{\ell}+ \sum_{\ell'=1}^{\ell} \lambda_{\ell'}[\vec{A}_{\ell}\vec{x}_{\ell}] \vec{S}_{\ell'} \prod_{i= \ell'-1}^{0} (\vec{I}- \lambda_{i}[\vec{A}_{\ell}\vec{x}_{\ell}] \vec{S}_{i}) \vec{x}_{\ell} \\
	&\overset{\mathclap{\eqref{eq: unnötig}}}{=} \,\Big(\vec{I} - \prod_{\ell'=\ell}^{0}(\vec{I}- \lambda_{\ell'}[\vec{A}_{\ell} \vec{x}_{\ell}]\vec{S}_{\ell'})\Big)\vec{x}_{\ell}.
	\end{split}
\end{align*}

With the identity $I$, let us define the operator $\SS_{\ell}^{\MG} : \XX_{\ell}^p \rightarrow \XX_{\ell}^p$ as
\begin{equation}\label{eq: mulitplicative schwarz operator}
	\SS_{\ell}^{\MG} v_{\ell} \coloneqq \Big(I- \prod_{\ell'= \ell}^{0} (I- \lambda_{\ell'}[\vec{A}_{\ell}\chi^p_{\ell}(v_{\ell})]\SS_{\ell'})\Big)v_{\ell}.
\end{equation}

\begin{remark}[Link multigrid and Schwarz methods]\label{remark: MS methods}
	The presented multigrid preconditioner~\eqref{eq: MG preconditioner} is related to multiplicative Schwarz methods. 
	To be precise, the operator $\SS_{\ell}^{\MG}$ has a multiplicative structure with levelwise operators $\lambda_{\ell'}\SS_{\ell'}$. 
	However, these operators are non-linear since the step-sizes $\lambda_{\ell'}$ depend on the input $v_{\ell} \in \XX_{\ell}^p$ via $\lambda_{\ell'} = \lambda_{\ell'}[\vec{A}_{\ell}\chi^p_{\ell}(v_{\ell})]$.
	If one were to omit the step-sizes $\lambda_{\ell'}$, then the operator $\SS_{\ell}^{\MG}$ would be a multiplicative Schwarz operator in the classical sense. 
	For more details on multiplicative Schwarz methods, we refer to~\cite{BPWX91,CW93}.
\end{remark}

The subsequent proposition shows the connection between the operator $\SS_{\ell}^{\MG}$ from~\eqref{eq: mulitplicative schwarz operator} and the preconditioner $\vec{B}_{\ell}^{\MG}$ from~\eqref{eq: MG preconditioner}.

\begin{proposition}[Multigrid functional/matrix representation]\label{lemma: connection between MG operator and matrix}
	Let $\SS_{\ell}^{\MG}$ be the multiplicative Schwarz operator defined in~\eqref{eq: mulitplicative schwarz operator} and $\vec{B}_{\ell}^{\MG}$ the multigrid preconditioner defined in~\eqref{eq: MG preconditioner}. 
	Then, it holds that
	\begin{equation}\label{eq: mulitplicative operator to matrix}
		\edual{\SS_{\ell}^{\MG}v_{\ell}}{w_{\ell}} = (\vec{B}_{\ell}^{\MG}[\vec{A}_{\ell}\vec{x}_{\ell}], \vec{y}_{\ell})_{\vec{A}_{\ell}} \quad \text{for all } v_{\ell},w_{\ell} \in \XX_{\ell}^p \text{ and } \vec{x}_{\ell} = \chi^p_{\ell}(v_{\ell}), \, \vec{y}_{\ell} = \chi^p_{\ell}(w_{\ell}).
	\end{equation} 
	Moreover, this implies
	\begin{equation}\label{eq: vector equality}
		\chi^p_{\ell}(\SS_{\ell}^{\MG}v_{\ell}) = \vec{B}_{\ell}^{\MG}[\vec{A}_{\ell}\chi^p_{\ell}(v_{\ell})] \quad \text{for all } v_{\ell} \in \XX_{\ell}^p.
	\end{equation}
	Lastly, for the update $\sigma_{\ell} = \Phi_{\ell}(u_{\ell})-u_{\ell}$ defined in Algorithm~\ref{algo: geometric MG}\eqref{eq: MG step 3}, we have
	\begin{equation}\label{eq: schwarz operator and MG}
		\SS_\ell^{\MG}(u_{\ell}^{\star} - u_{\ell}) = \sigma_{\ell} \quad \text{for all } u_{\ell} \in \XX_{\ell}^p.
	\end{equation}
\end{proposition}

\begin{proof}[\textbf{Proof}]
	Let $\ell \in \N_0$ be fixed.
	Let $v_{\ell},w_{\ell} \in \XX_{\ell}^p$ and $\vec{x}_{\ell} = \chi^p_{\ell}(v_{\ell})$, $\vec{y}_{\ell} = \chi^p_{\ell}(w_{\ell})$.
	We show that
	\begin{equation}\label{eq: product translation}
		\edualbig{\prod_{\ell'= \ell}^{0}(I-\lambda_{\ell'}[\vec{A}_{\ell}\chi^p_{\ell}(v_{\ell})]\SS_{\ell'})v_{\ell}}{{w_{\ell}}} = \Big(\prod_{\ell'= \ell}^{0}(\vec{I}-\lambda_{\ell'}[\vec{A_{\ell}x}_{\ell}]\vec{S}_{\ell'})\vec{x}_{\ell}, \vec{y}_{\ell} \Big)_{\vec{A}_{\ell}} 
	\end{equation}
	by induction on $\ell'$, i.e., the induction hypothesis reads as
	\begin{equation}\label{eq: induction hypothesis}
		\edualbig{\prod_{i=\ell'}^{0}(I-\lambda_{i}[\vec{A}_{\ell}\chi^p_{\ell}(v_{\ell})]\SS_{i})v_{\ell}}{w_{\ell}} = \Big(\prod_{i=\ell'}^{0}(\vec{I}-\lambda_{i}[\vec{A_{\ell}x}_{\ell}]\vec{S}_{i})\vec{x}_{\ell}, \vec{y}_{\ell} \Big)_{\vec{A}_{\ell}} 
	\end{equation}
	for $\ell' \in \{0, \dots, \ell-1\}$.
	From Lemma~\ref{lemma: levelwise operator and Matrix}, we immediately obtain the base case $\ell' = 0$.
	Since the operators $\SS_{\ell'+1}$ from~\eqref{eq: levelwise operators} and $\vec{S}_{\ell'+1}$ from~\eqref{eq: levelwise matrices} are symmetric with respect to the scalar products $\edual{\cdot}{\!\cdot}$ and $(\cdot, \cdot)_{\vec{A}_{\ell}}$, respectively, it follows from the induction hypothesis~\eqref{eq: induction hypothesis} that
	\begin{align*}
		&\edualbig{\prod_{i= \ell' +1}^{0}(I-\lambda_{i}[\vec{A}_{\ell}\chi^p_{\ell}(v_{\ell})]\SS_{i})v_{\ell}}{w_{\ell}} \\
		&\qquad = \, \edualbig{\prod_{i= \ell'}^{0}(I-\lambda_{i}[\vec{A}_{\ell}\chi^p_{\ell}(v_{\ell})]\SS_{i}) v_{\ell}}{(I-\lambda_{\ell'+1}[\vec{A}_{\ell}\chi^p_{\ell}(v_{\ell})]\SS_{\ell'+1})w_{\ell}} \\
		& \qquad \eqreff*{eq: induction hypothesis}{=} \, \Big( \prod_{i= \ell'}^{0}(\vec{I}-\lambda_{i}[\vec{A_{\ell}x}_{\ell}]\vec{S}_{i}) \vec{x}_{\ell}, \widetilde{\vec{y}}_{\ell} \Big)_{\vec{A}_{\ell}},
	\end{align*}
	where $\widetilde{\vec{y}}_{\ell} = \chi^p_{\ell}\big((I-\lambda_{\ell'+1}[\vec{A}_{\ell}\chi^p_{\ell}(v_{\ell})]\SS_{\ell'+1})w_{\ell}\big)$.
	Let $\widetilde{v}_{\ell} \in \XX_{\ell}^p$ denote the discrete function such that $\prod_{i= \ell'}^{0}(\vec{I}-\lambda_{i}[\vec{A_{\ell}x}_{\ell}]\vec{S}_{i}) \vec{x}_{\ell} = \chi^p_{\ell}(\widetilde{v}_{\ell})$. 
	Utilizing the connection~\eqref{eq: levelwise operator and Matrix} and the symmetry of $\vec{S}_{\ell'+1}$, we get
	\begin{align*}
		\Big( \prod_{i=\ell'}^{0}(\vec{I}-\lambda_{i}[\vec{A_{\ell}x_{\ell}}]\vec{S}_{i}) \vec{x}_{\ell}, \widetilde{\vec{y}}_{\ell}\Big)_{\vec{A}_{\ell}} &= \:\edual{\widetilde{v}_{\ell}}{(I-\lambda_{\ell'+1}[\vec{A}_{\ell}\chi^p_{\ell}(v_{\ell})]\SS_{\ell'+1})w_{\ell}} \\
		&\overset{\mathclap{\eqref{eq: levelwise operator and Matrix}}}{=} \: \Big(\prod_{i=\ell'}^{0}(\vec{I}-\lambda_{i}[\vec{A_{\ell}x_{\ell}}]\vec{S}_{i})\vec{x}_{\ell}, (\vec{I}-\lambda_{\ell'+1}[\vec{A}_{\ell}\vec{x}_{\ell}]\vec{S}_{\ell'+1}) \vec{y}_{\ell} \Big)_{\vec{A}_{\ell}} \\
		&= \:\Big(\prod_{i= \ell' + 1}^{0}(\vec{I}-\lambda_{i}[\vec{A_{\ell}x_{\ell}}]\vec{S}_{i})\vec{x}_{\ell}, \vec{y}_{\ell} \Big)_{\vec{A}_{\ell}}.
	\end{align*}
This concludes the induction step and thus proves~\eqref{eq: product translation}.
	Since $\edual{v_{\ell}}{w_{\ell}} = (\vec{x}_{\ell},\vec{y}_{\ell})_{\vec{A}_{\ell}}$, we obtain
	\begin{align*}
		\edual{\SS_{\ell}^{\MG}v_{\ell}}{w_{\ell}} &\, \eqreff*{eq: mulitplicative schwarz operator}{=} \, \edual{v_{\ell}}{w_{\ell}} - \edualbig{\prod_{\ell'= \ell}^{0}(I-\lambda_{\ell'}[\vec{A}_{\ell}\chi^p_{\ell}(v_{\ell})]\SS_{\ell'})v_{\ell}}{w_{\ell}} \\
		&\, \eqreff*{eq: product translation}{=} \, (\vec{x}_{\ell},\vec{y}_{\ell})_{\vec{A}_{\ell}} - \Big(\prod_{\ell'= \ell}^{0}(\vec{I}-\lambda_{\ell'}[\vec{A_{\ell}x_{\ell}}]\vec{S}_{\ell'})\vec{x}_{\ell}, \vec{y}_{\ell} \Big)_{\vec{A}_{\ell}} \, \eqreff*{eq: MG preconditioner}{=} \, (\vec{B}_{\ell}^{\MG}[\vec{A}_{\ell}\vec{x}_{\ell}], \vec{y}_{\ell})_{\vec{A}_{\ell}}.
	\end{align*}
	From~\eqref{eq: mulitplicative operator to matrix}, we immediately have
	\begin{equation*}
		(\chi^p_{\ell}(\SS_{\ell}^{\MG}v_{\ell}), \vec{y}_{\ell})_{\vec{A}_{\ell}} = \edual{\SS_{\ell}^{\MG}v_{\ell}}{w_{\ell}} \overset{\eqref{eq: mulitplicative operator to matrix}}{=} (\vec{B}_{\ell}^{\MG}[\vec{A}_{\ell}\vec{x}_{\ell}], \vec{y}_{\ell})_{\vec{A}_{\ell}}.
	\end{equation*}
	Since this holds for every $\vec{y}_{\ell} \in \R^{N_{\ell}^p}$, we obtain~\eqref{eq: vector equality}.
	Let $u_{\ell} \in \XX_{\ell}^p$, $\sigma_{\ell} = \Phi_{\ell}(u_{\ell})-u_{\ell}$ be the update constructed by Algorithm~\ref{algo: geometric MG}, and $\vec{s}_{\ell}  = \chi^p_{\ell}(\sigma_{\ell})$. In~\eqref{eq: MG update}, we derived that $\vec{s}_{\ell} = \vec{B}_{\ell}^{\MG}[\vec{r}_{\ell}]$, where $\vec{r}_{\ell} = \vec{A}_{\ell}(\vec{x}_{\ell}^{\star}-\vec{x}_{\ell})$, with $\vec{x}_{\ell}^{\star} = \chi_{\ell}^p(u_{\ell}^{\star})$ and $\vec{x}_{\ell} = \chi_{\ell}^p(u_{\ell})$.
	Finally,~\eqref{eq: vector equality} implies $\vec{B}_{\ell}^{\MG}[\vec{A}_{\ell}(\vec{x}_{\ell}^{\star}-\vec{x}_{\ell})]= \chi^p_{\ell}(\SS_{\ell}^{\MG}(u_{\ell}^{\star}-u_{\ell}))$ and hence $\SS_{\ell}^{\MG}(u_{\ell}^{\star}-u_{\ell}) = \sigma_{\ell}$.
	This concludes the proof.
\end{proof}

Recall the following result on Algorithm~\ref{algo: geometric MG} from~\cite[Theorem 2.5]{IMPS24}, where the local dependence on the diffusion coefficient is proved in~\cite{Pau25}.

\begin{proposition}[Robust contraction of multigrid~\cite{IMPS24,Pau25}]\label{lemma: MG contraction}
	Let $u_{\ell}^{\star} \in \XX_{\ell}^p$ be the solution of~\eqref{eq: discrete problem} and $u_{\ell} \in \XX_{\ell}^p$ an approximation. 
	Then, there exists a constant $q \in (0,1)$ such that
	\begin{equation}\label{eq: MG contraction}
		\enorm{u_{\ell}^{\star}-\Phi_{\ell}(u_{\ell})} \leq q \, \enorm{u_{\ell}^{\star}- u_{\ell}} \quad \text{for all } u_{\ell} \in \XX_{\ell}^p.
	\end{equation}	
	The factor $q$ depends only on the space dimension $d$, the initial mesh $\TT_0$, the $\gamma$-shape regularity~\eqref{eq: shape regularity}, and the local constants $\Cloc$ and $\CLoc$, which are defined as
	\begin{equation}\label{eq: local constant 2}
		\Cloc \coloneqq \max \Big\{\sup_{z_0 \in \VV_0}\frac{\max_{T \subseteq \overline{\omega_{0}^2(z_0)}}\|\div(\vec{K})\|_{L^{\infty}(T)}}{\inf_{y \in \omega_{0}^2(z_0)}\lambda_{\min}(\vec{K}(y))}, \sup_{z_0 \in \VV_0} \frac{\sup_{y \in \omega_{0}^2(z_0)} \lambda_{\max}(\vec{K}(y))}{\inf_{y \in \omega_{0}^2(z_0)}\lambda_{\min}(\vec{K}(y))}\Big\}.
	\end{equation}
	and
	\begin{equation}\label{eq: local constant 1}
		\CLoc \coloneqq \sup_{z \in \VV_0} \frac{\sup_{y \in \omega_0^3(z)} \lambda_{\max}(\vec{K}(y))}{\inf_{y \in \omega_0^3(z)}\lambda_{\min}(\vec{K}(y))}. \qed
	\end{equation}
\end{proposition}

\begin{remark}
	Note that the analytical contraction factors $q$ in Proposition~\ref{lemma: MG contraction} for the standalone geometric MG and in Theorem~\ref{theorem: main result} for GPCG with non-linear and non-symmetric MG preconditioner are the same, i.e., GPCG can only improve contraction in practice. 
	Indeed this is what we observe in our numerical experiments; see Section~\ref{sec: numerics}.
\end{remark}

\begin{remark}\label{remark: linearized MG}
	Instead of the optimal step-sizes $\lambda_{\ell'}$ defined in Algorithm~\ref{algo: geometric MG}\eqref{eq: MG step 2}, one can alternatively use the fixed step-size $\lambda \coloneqq (d+1)^{-1}$ for $\ell' \geq 1$ in Algorithm~\ref{algo: geometric MG}.
	In this case the multigrid preconditioner $\vec{B}_{\ell}^{\MG}$ becomes linear and we denote it with $\vec{B}_{\ell}^{\LMG}$. 
	All results of this section remain valid.
	In particular, Theorem~\ref{theorem: general result} implies that $\vec{B}_{\ell}^{\LMG}$ fulfills assumption~\eqref{eq: preconditioner assumption}, i.e.,
	\begin{equation}\label{eq: linearized MG contracion}
		|(\vec{I}-\vec{B}_{\ell}^{\LMG}\vec{A}_{\ell})\vec{x}_{\ell}|_{\vec{A}_{\ell}} \leq q \, |\vec{x}_{\ell}|_{\vec{A}_{\ell}} \quad \text{for all } \vec{x}_{\ell} \in \R^{N_{\ell}^p},
	\end{equation}
	 where $\vec{B}_{\ell}^{\LMG}\vec{A}_{\ell} = \vec{I} - \big(\prod_{\ell'=\ell}^{1} (\vec{I}-\lambda \, \vec{S}_{\ell'}) \big) (\vec{I}-\vec{S}_0)$ and $q \in (0,1)$ is the constant from~\eqref{eq: MG contraction}.
\end{remark}

\subsection{Proof of Theorem~\ref{theorem: main result}}\label{subsec: prove GPCG MG}
	We show that the multigrid preconditioner $\vec{B}_{\ell}^{\MG}$ fulfills the assumption of Theorem~\ref{theorem: general result}, i.e., contraction~\eqref{eq: contraction of operator} of the operator $\SS_{\ell}^{\MG}$.
	Let $u_{\ell} \in \XX_{\ell}^p$.
	With~\eqref{eq: schwarz operator and MG}, we can rewrite the estimate~\eqref{eq: MG contraction} as
	\begin{equation*}
		\enorm{(I- \SS_{\ell}^{\MG})(u_{\ell}^{\star}-u_{\ell})}  = \enorm{u_{\ell}^{\star}-(u_{\ell}+\sigma_{\ell})}= \enorm{u_{\ell}^{\star}-\Phi_{\ell}(u_{\ell})} \eqreff{eq: MG contraction}{\leq} q \, \enorm{u_{\ell}^{\star}-u_{\ell}}.
	\end{equation*}
	Since this holds for all $u_{\ell} \in \XX_{\ell}^p$, we also have
	\begin{equation*}
		\enorm{(I- \SS_{\ell}^{\MG})v_{\ell}} \leq q \, \enorm{v_{\ell}} \quad \text{for all } v_{\ell} \in \XX_{\ell}^p.
	\end{equation*}
	Applying Theorem~\ref{theorem: general result} concludes the proof.
	\qed

\section{Optimal linear and symmetric multigrid preconditioner}\label{sec: symmetric MG}

In this section, we formulate and analyze a linear and symmetric multigrid preconditioner $\vec{B}^{\SMG}_{\ell}$.
First, we linearize the multigrid of~\cite{IMPS24} by replacing the optimal step-sizes $\lambda_{\ell'}$ with the constant $\lambda = (d+1)^{-1}$ for all $\ell' \in \{1, \ldots, \ell\}$, as already described in Remark~\ref{remark: linearized MG}.
Second, we symmetrize the approach by adding suitable pre-smoothing steps to Algorithm~\ref{algo: geometric MG}.

\subsection{Preconditioner and corresponding main result}
We build on the notation already introduced in Section~\ref{sec: GPCG}.
The additional superscript is introduced to distinguish algorithmically the pre-smoothing ($\downarrow$) and post-smoothing ($\uparrow$) in the V-cycle.

\begin{algorithm}[V-cycle of symmetric and linear multigrid method]\label{algo: symmetric MG}
	\textbf{Input:} Current approximation $u_{\ell} \in \XX_{\ell}^p$, triangulations $\{\TT_{\ell'}\}_{\ell'=0}^\ell$, and polynomial degree $p \geq 1$. \\
	Follow the steps~\eqref{item: sMG step 1}--\eqref{item: sMG step 4}:
	\begin{enumerate}[label=\textnormal{(\roman*)}, ref = \textnormal{\roman*}]
		\item \label{item: sMG step 1}\textbf{High-order correction:} For all $z \in \VV_{\ell}$, compute $\rho^{\downarrow}_{\ell,z} \in \XX_{\ell,z}^p$ such that
		\begin{equation*}
			\edual{\rho^{\downarrow}_{\ell,z}}{v_{\ell,z}} = R_{\ell}(v_{\ell,z}) \quad \text{for all } v_{\ell,z} \in \XX_{\ell,z}^p.
		\end{equation*}
		Define $\rho^{\downarrow}_{\ell} \coloneqq \sum_{z \in \VV_{\ell}} \rho^{\downarrow}_{\ell,z}$ and $\sigma^{\downarrow}_{\ell} \coloneqq \lambda \rho^{\downarrow}_{\ell} = (d+1)^{-1} \rho^{\downarrow}_{\ell}$.

		\item \label{item: sMG step 2}\textbf{Lowest-order correction:} For all intermediate levels $\ell' = \ell -1, \dots, 1$ and all $z \in \VV_{\ell'}^+$, compute $\rho^{\downarrow}_{\ell',z} \in \XX_{\ell',z}^1$ such that
		\begin{equation*}
			\edual{\rho^{\downarrow}_{\ell',z}}{v_{\ell',z}} = R_{\ell}(v_{\ell',z}) - \edual{\sigma^{\downarrow}_{\ell'+1}}{v_{\ell',z}} \quad \text{for all } v_{\ell',z} \in \XX_{\ell',z}^1.
		\end{equation*}
		Define $\rho^{\downarrow}_{\ell'} \coloneqq \sum_{z \in \VV_{\ell'}^+} \rho^{\downarrow}_{\ell',z}$ and $\sigma^{\downarrow}_{\ell'} \coloneqq \sigma^{\downarrow}_{\ell'+1} + \lambda \rho^{\downarrow}_{\ell'} = \sigma^{\downarrow}_{\ell'+1} + (d+1)^{-1} \rho^{\downarrow}_{\ell'}$.

		\item \label{item: sMG step 3}\textbf{Coarse level solve:}  Compute $\rho^{\uparrow}_0 \in \XX_0^1$ such that 
		\begin{equation}
			\edual{\rho^{\uparrow}_0}{v_0} = R_{\ell}(v_0) - \edual{\sigma_1^{\downarrow}}{v_0}\quad \text{for all } v_0 \in \XX_0^1. 
		\end{equation}
		Define $\sigma^{\uparrow}_0 \coloneqq \sigma^{\downarrow}_1 + \rho^{\uparrow}_0$.

		\item \textbf{Lowest-order correction:} For all intermediate levels $\ell' = 1, \dots, \ell-1$ and all $z \in \VV_{\ell'}^+$, compute $\rho^{\uparrow}_{\ell',z} \in \XX_{\ell',z}^1$ such that
		\begin{equation}
			\edual{\rho^{\uparrow}_{\ell',z}}{v_{\ell',z}} = R_{\ell}(v_{\ell',z}) - \edual{\sigma^{\uparrow}_{\ell'-1}}{v_{\ell',z}} \quad \text{for all } v_{\ell',z} \in \XX_{\ell',z}^1.
		\end{equation}
		Define $\rho^{\uparrow}_{\ell'} \coloneqq \sum_{z \in \VV_{\ell'}^+} \rho^{\uparrow}_{\ell',z}$ and $\sigma^{\uparrow}_{\ell'} \coloneqq \sigma^{\uparrow}_{\ell'-1} + \lambda \rho^{\uparrow}_{\ell'} = \sigma^{\uparrow}_{\ell'-1} + (d+1)^{-1} \rho^{\uparrow}_{\ell'}$.

		\item \label{item: sMG step 4}\textbf{High-order correction:} For all $z \in \VV_{\ell}$, compute $\rho^{\uparrow}_{\ell,z} \in \XX_{\ell,z}^p$ such that
		\begin{equation*}
			\edual{\rho^{\uparrow}_{\ell,z}}{v_{\ell,z}} = R_{\ell}(v_{\ell,z}) - \edual{\sigma^{\uparrow}_{\ell-1}}{v_{\ell,z}} \quad \text{for all } v_{\ell,z} \in \XX_{\ell,z}^p.
		\end{equation*}
		Define $\rho^{\uparrow}_{\ell} \coloneqq \sum_{z \in \VV_{\ell}} \rho^{\uparrow}_{\ell,z}$ and $\sigma_{\ell} \coloneqq\sigma^{\uparrow}_{\ell} \coloneqq \sigma^{\uparrow}_{\ell-1} + \lambda \rho^{\uparrow}_{\ell} = \sigma^{\uparrow}_{\ell-1} + (d+1)^{-1} \rho^{\uparrow}_{\ell}$.
	\end{enumerate}
	\textbf{Output:} Improved approximation $\widetilde{\Phi}_{\ell}(u_{\ell}) \coloneqq u_{\ell} + \sigma_{\ell}$.
\end{algorithm}

\begin{remark}[Computational complexity of Algorithm~\ref{algo: symmetric MG}]\label{remark: cost sMG}
	With the same arguments as in Remark~\ref{rem: complexity of geometric MG}, we obtain that the overall computational cost of Algorithm~\ref{algo: symmetric MG} is of order $\OO(\# \TT_{\ell})$.
	Note that the calculation of the optimal step-sizes $\lambda_{\ell'}$ is not applicable in Algorithm~\ref{algo: symmetric MG}.
\end{remark}

Similarly to~\eqref{eq: levelwise preconditioner} in Section~\ref{sec: geometric multigrid}, we define the matrices $\vec{B}^{\downarrow}_{\ell'}$ for $\ell' \in \{1,\dots, \ell\}$ and $\vec{B}^{\uparrow}_{\ell'}$ for $\ell' \in \{0,\dots, \ell\}$ by
\begin{align*}
	\vec{B}^{\downarrow}_{\ell'} &\coloneqq \vec{I}^+_{\ell'}(\vec{D}_{\ell'}^+)^{-1}(\vec{I}^+_{\ell'})^{T} \prod_{i=\ell'+1}^{\ell}(\vec{I}-\lambda \, \vec{A}_{\ell}\vec{I}_{i}^+(\vec{D}_{i}^+)^{-1}(\vec{I}_i^+)^{T}), \\
	\vec{B}^{\uparrow}_{0} &\coloneqq \vec{I}^+_{0}(\vec{A}_0)^{-1}(\vec{I}^+_{0})^{T} \prod_{i=1}^{\ell}(\vec{I}-\lambda \, \vec{A}_{\ell}\vec{I}_{i}^+(\vec{D}_{i}^+)^{-1}(\vec{I}_i^+)^{T}) \quad \text{and} \\
	\vec{B}^{\uparrow}_{\ell'} &\coloneqq \vec{I}^+_{\ell'}(\vec{D}_{\ell'}^+)^{-1}(\vec{I}^+_{\ell'})^{T}\prod_{i= \ell' -1}^{1}(\vec{I}-\lambda \, \vec{A}_{\ell}\vec{I}_{i}^+(\vec{D}_{i}^+)^{-1}(\vec{I}_{i}^+)^{T})\prod_{i=0}^{\ell}(\vec{I}-\lambda \, \vec{A}_{\ell}\vec{I}_{i}^+(\vec{D}_{i}^+)^{-1}(\vec{I}_i^+)^{T}),
\end{align*}
where we set $\vec{D}_0^+ \coloneqq \lambda \,\vec{A}_{0}$ and $(\vec{D}^+_{\ell})^{-1} \coloneqq \sum_{z \in \VV_{\ell}} \vec{I}_{\ell,z}^p (\vec{A}_{\ell,z}^p)^{-1}(\vec{I}_{\ell,z}^p)^T$ for easier notation.

Let $u_{\ell}$ be the current approximation of $u_{\ell}^{\star}$ and $\vec{x}_{\ell} = \chi_{\ell}^p[u_{\ell}], \vec{x}_{\ell}^{\star} = \chi_{\ell}^p[u_{\ell}^{\star}] \in \R^{N_{\ell}^p}$.
Moreover, we set $\vec{r}_{\ell} = \vec{b}_{\ell} - \vec{A}_{\ell} \vec{x}_{\ell} = \vec{A}_{\ell}(\vec{x}^{\star}_{\ell}-\vec{x}_{\ell})$. Along the lines of~\eqref{eq: levelwise matrices}--\eqref{eq: MG update} in Section~\ref{sec: hp multigrid}, one also shows that the coefficient vectors $\vec{s}^{\downarrow}_{\ell'} = \chi_{\ell}^p (\rho^{\downarrow}_{\ell'}), \vec{s}^{\uparrow}_{\ell'} = \chi_{\ell}^p (\rho^{\uparrow}_{\ell'}) \in \R^{N_{\ell}^p}$  are given by
\begin{equation*}
	\vec{s}^{\downarrow}_{\ell'} = \vec{B}^{\downarrow}_{\ell'}\vec{r}_{\ell}, \quad \vec{s}^{\uparrow}_0 = \vec{B}^{\uparrow}_0 \vec{r}_{\ell} \quad \text{and} \quad \vec{s}^{\uparrow}_{\ell'} = \vec{B}^{\uparrow}_{\ell'}\vec{r}_{\ell}.
\end{equation*}
Proceeding analogously as in Section~\ref{sec: geometric multigrid}, we obtan that the total error update $\sigma_{\ell}$ satisfies
\begin{equation*}
	\chi_{\ell}^p (\sigma_{\ell}) =\vec{s} = \lambda \sum_{\ell'=1}^{\ell} \vec{s}^{\downarrow}_{\ell'} + \vec{s}^{\uparrow}_0 + \lambda \sum_{\ell'=1}^{\ell} \vec{s}^{\uparrow}_{\ell'} = \Big(\lambda\sum_{\ell'=1}^{\ell} \vec{B}^{\downarrow}_{\ell'} + \vec{B}^{\uparrow}_0 +\lambda \sum_{\ell'=1}^{\ell} \vec{B}^{\uparrow}_{\ell'} \Big)\vec{r}_{\ell}.
\end{equation*}
With these preparations, we define the linear and symmetric multigrid preconditioner $\vec{B}^{\SMG}_{\ell}$
\begin{equation}\label{eq: sMG preconditioner}
	\vec{B}^{\SMG}_{\ell} \coloneqq \lambda\sum_{\ell'=1}^{\ell} \vec{B}^{\downarrow}_{\ell'} + \vec{B}^{\uparrow}_{0} + \lambda \sum_{\ell'=1}^{\ell} \vec{B}^{\uparrow}_{\ell'} \in \R^{N_{\ell}^p \times N_{\ell}^p}.
\end{equation}
Lemma~\ref{lemma: unnötig} yields the representation
\begin{equation}\label{eq: symmetric multigrid preconditioned matrix}
	\vec{B}^{\SMG}_{\ell} \vec{A}_{\ell} = \vec{I} - \prod_{\ell'= \ell}^{1} (\vec{I}-\lambda \, \vec{S}_{\ell'})(\vec{I}-\vec{S}_0)\prod_{\ell' = 1}^{\ell} (\vec{I}-\lambda \, \vec{S}_{\ell'}).
\end{equation}

The subsequent theorem shows that $\textnormal{cond}_{\vec{A}_{\ell}}(\vec{B}^{\SMG}_{\ell}\vec{A}_{\ell})$ is bounded independently of $h$ and $p$. 
Thus, Proposition~\ref{theorem: CG convergence} yields that PCG with preconditioner $\vec{B}_{\ell}^{\SMG}$ contracts $h$- and $p$-robustly.
The proof is postponed to Section~\ref{subsec: proof sMG}.
With Remark~\ref{remark: cost sMG}, it follows that $\vec{B}_{\ell}^{\SMG}$ is indeed optimal.

\begin{theorem}[PCG with linear and symmetric MG preconditioner]\label{corollary: symmetric MG optimality}
	The symmetric multigrid preconditioner $\vec{B}^{\SMG}_{\ell}$ defined in~\eqref{eq: sMG preconditioner} is SPD and optimal, i.e., there holds
	\begin{equation}\label{eq:sMG spectral bound}
		\operatorname{cond}_2((\vec{B}_{\ell}^{\SMG})^{1/2}\vec{A}_{\ell}(\vec{B}_{\ell}^{\SMG})^{1/2}) = \textnormal{cond}_{\vec{A}_{\ell}}(\vec{B}^{\SMG}_{\ell}\vec{A}_{\ell}) \leq \frac{1 + q^2}{1-q^2},
	\end{equation}
	with the contraction factor $0 < q <1$ from Proposition~\ref{lemma: MG contraction}. 
	For the iterates $\vec{x}_{\ell}^k$ generated by PCG employing the preconditioner~\eqref{eq: sMG preconditioner} with initial guess $\vec{x}_{\ell}^0$, there holds
	\begin{equation}\label{eq:sMG contraction}
		|\vec{x}_{\ell}^\star- \vec{x}_\ell^{k+1}|_{\vec{A}_{\ell}} \leq \Big(1- \frac{1-q^2}{1+q^2}\Big)^{1/2} |\vec{x}_{\ell}^\star- \vec{x}_\ell^k|_{\vec{A}_\ell} \quad \text{for all } k \in \N_0.
	\end{equation}
	For $u_{\ell}^{\star} = \sum_{j=1}^{N_{\ell}^p} (\vec{x}_{\ell}^{\star})_j \varphi^p_{\ell,j}$ and $u_{\ell}^{k} = \sum_{j=1}^{N_{\ell}^p} (\vec{x}_{\ell}^{k})_j \varphi^p_{\ell,j}$, this directly translates to 
	\begin{equation*}
		 \enorm{u_{\ell}^{\star}- u_{\ell}^{k+1}} \leq \Big(1- \frac{1-q^2}{1+q^2}\Big)^{1/2} \, \enorm{u_{\ell}^{\star} - u_{\ell}^k} \quad \text{for all } k \in \N_0. 
	\end{equation*}
\end{theorem}

The following lemma provides a connection between the notion of symmetrizing the multigrid via pre-smoothing and symmetrizing matrices via their transpose.

\begin{lemma}[Symmetrized MG is a symmetric preconditioner]\label{lemma: symmetric multigrid}
	It holds that $\vec{B}^{\SMG}_{\ell}\vec{A}_{\ell}$ is symmetric with respect to $(\cdot, \cdot )_{\vec{A}_{\ell}}$ and hence $\vec{B}^{\SMG}_{\ell}$ is a symmetric matrix. 
	Let us define the error propagation matrix
	\begin{equation}\label{eq: error propagation martix}
		\vec{E}_{\ell} \coloneqq \prod_{\ell'= \ell}^{1} (\vec{I}-\lambda \, \vec{S}_{\ell'}) (\vec{I}-\vec{S}_0) = \vec{I}-\vec{B}_{\ell}^{\LMG}\vec{A}_{\ell},
	\end{equation}
	where $\vec{B}_{\ell}^{\LMG}$ is defined in Remark~\ref{remark: linearized MG}.
	Then, it follows that
	\begin{equation}\label{eq: SMG matrix symmetrization}
		\vec{B}^{\SMG}_{\ell} \vec{A}_{\ell} = \vec{I} - \vec{E}_{\ell} \vec{E}_{\ell}^{T_{\vec{A}_{\ell}}}.
	\end{equation}
	where $T_{\vec{A}_{\ell}}$ denotes the transpose with respect to the inner product $(\cdot, \cdot)_{\vec{A}_{\ell}}$.
\end{lemma}

\begin{proof}[\textbf{Proof}]
	Let $\vec{A}$ be an SPD matrix and $\vec{N}$ and $\vec{M}$ be any two matrices of the same size as $\vec{A}$.
	Then, it holds that $(\vec{N}\vec{M})^{T_{\vec{A}}} = \vec{M}^{T_{\vec{A}}}\vec{N}^{T_{\vec{A}}}.$
	Lemma~\ref{lemma: levelwise operator and Matrix} yields that the levelwise matrices $\vec{I}-\lambda \vec{S}_{\ell'}$ and $\vec{I}-\vec{S}_0$ are symmetric with respect to $(\cdot, \cdot)_{\vec{A}_{\ell}}$. 
	With~\eqref{eq: symmetric multigrid preconditioned matrix}, we conclude that $\vec{B}^{\SMG}_{\ell}\vec{A}_{\ell}$ is symmetric with respect to $(\cdot, \cdot)_{\vec{A}_{\ell}}$.
	Due to the symmetry of the levelwise matrices, we can write
	\begin{equation*}
		\vec{E}_{\ell} \vec{E}_{\ell}^{T_{\vec{A}_{\ell}}} = \prod_{\ell'= \ell}^{1} (\vec{I}-\lambda \, \vec{S}_{\ell'}) (\vec{I}-\vec{S}_0) (\vec{I}-\vec{S}_0)\prod_{\ell' = 1}^{\ell} (\vec{I}-\lambda \, \vec{S}_{\ell'}).
	\end{equation*}
	Since $\vec{S}_0$ is the matrix representation of the lowest-order Galerkin projection $\PP_0^1$, we get
	\begin{equation*}
		(\vec{I} - \vec{S}_0)(\vec{I}-\vec{S}_0) = \vec{I} - \vec{S}_0\,
	\end{equation*}
	and~\eqref{eq: SMG matrix symmetrization} follows from~\eqref{eq: symmetric multigrid preconditioned matrix}.
	This concludes the proof.
\end{proof}

\subsection{Proof of Theorem~\ref{corollary: symmetric MG optimality}}\label{subsec: proof sMG}
	For any SPD matrix $\vec{A}$ and any square matrix $\vec{M}$ satisfy
	\begin{equation}\label{eq: transpose connection}
		(\vec{x}, \vec{M}^T \vec{y})_2 = (\vec{M}\vec{x}, \vec{y})_2 = (\vec{M}\vec{x}, \vec{A}^{-1}\vec{y})_{\vec{A}} = (\vec{x}, \vec{M}^{T_{\vec{A}}}\vec{A}^{-1}\vec{y})_{\vec{A}} = (\vec{x}, \vec{A}\vec{M}^{T_{\vec{A}}}\vec{A}^{-1}\vec{y})_2 
 	\end{equation}
	for all $\vec{x},\vec{y} \in \R^{N_{\ell}^p}$ and hence $\vec{M}^T = \vec{A} \vec{M}^{T_{\vec{A}}}\vec{A}^{-1}$.
	Furthermore, we have
	\begin{equation}\label{eq: operator norm identity}
		|\vec{M}|_{\vec{A}} = \max_{|\vec{x}|_{\vec{A}}=1} |\vec{M}\vec{x}|_{\vec{A}} = \max_{|\vec{y}|_2 = 1} |\vec{A}^{1/2}\vec{M}\vec{A}^{-1/2}\vec{y}|_2 = |\vec{A}^{1/2}\vec{M}\vec{A}^{-1/2}|_2.
	\end{equation} 
	Recall the familiar identity
	\begin{equation}\label{eq: euclidean norm identity}
		|\vec{M}|_2^2 = |\vec{M}\vec{M}^T|_2 = |\vec{M}^T|_2^2.
	\end{equation}
	Altogether, it follows that
	\begin{align*}
		|\vec{M}|_{\vec{A}}^2 &\overset{\eqref{eq: operator norm identity}}{=} |\vec{A}^{1/2}\vec{M}\vec{A}^{-1/2}|_2^2 \overset{\eqref{eq: euclidean norm identity}}{=} |\vec{A}^{1/2}\vec{M}\vec{A}^{-1}\vec{M}^T\vec{A}^{1/2}|_2^2 \overset{\eqref{eq: transpose connection}}{=} |\vec{A}^{1/2}\vec{M}\vec{M}^{T_{\vec{A}}}\vec{A}^{-1/2}|_2 \overset{\eqref{eq: operator norm identity}}{=} |\vec{M}\vec{M}^{T_{\vec{A}}}|_{\vec{A}}
	\end{align*}
	as well as
	\begin{equation*}
		|\vec{M}|_{\vec{A}} \overset{\eqref{eq: operator norm identity}}{=} |\vec{A}^{1/2}\vec{M}\vec{A}^{-1/2}|_2 \overset{\eqref{eq: euclidean norm identity}}{=} |\vec{A}^{1/2}\vec{M}^T\vec{A}^{-1/2}|_2 = |\vec{M}^{T_{\vec{A}}}|_{\vec{A}}.
	\end{equation*}
	This proves
	\begin{equation}\label{eq: norm equality}
		|\vec{M}|^2_{\vec{A}} = |\vec{M}\vec{M}^{T_{\vec{A}}}|_{\vec{A}} = |\vec{M}^{T_{\vec{A}}}|^2_{\vec{A}}.
	\end{equation}
	Symmetry of $\vec{B}_{\ell}^{\SMG}$ follows from Lemma~\ref{lemma: symmetric multigrid}. 
	To show that $\vec{B}_{\ell}^{\SMG}$ is positive definite, note that~\eqref{eq: SMG matrix symmetrization}, \eqref{eq: norm equality}, and~\eqref{eq: linearized MG contracion} imply that 
	\begin{align*}
		(\vec{B}^{\SMG}_{\ell}\vec{A}_{\ell}\vec{x}_{\ell}, \vec{x}_{\ell})_{\vec{A}_{\ell}} \, &\eqreff*{eq: SMG matrix symmetrization}{=} \, ((\vec{I}-\vec{E}_{\ell}\vec{E}_{\ell}^{T_{\vec{A}_{\ell}}})\vec{x}_{\ell}, \vec{x}_{\ell})_{\vec{A}_{\ell}} \eqreff{eq: norm equality}= (\vec{x}_{\ell},\vec{x}_{\ell})_{\vec{A}_{\ell}} - |\vec{E}_{\ell}\vec{x}_{\ell}|^2_{\vec{A}_{\ell}} \\
		& \, \eqreff*{eq: linearized MG contracion}{\geq} \, (1-q) |\vec{x}_{\ell}|^2_{\vec{A}_{\ell}} > 0 \quad \text{for all } \vec{x}_{\ell} \in \R^{N_{\ell}^p} \backslash \{0\}.
	\end{align*}
	Substituting $\vec{y}_{\ell} = \vec{A}_{\ell}\vec{x}_{\ell}$ yields $(\vec{B}^{\SMG}_{\ell}\vec{y}_{\ell},\vec{y}_{\ell})_2 >0$ for all $\vec{y}_{\ell} \neq 0$. 
	Hence, $\vec{B}^{\SMG}_{\ell}$ is SPD.
	Using identity~\eqref{eq: SMG matrix symmetrization}, the Neumann series, 
	identity~\eqref{eq: norm equality}, and~\eqref{eq: linearized MG contracion}, we get
	\begin{align*}
		\textnormal{cond}_{\vec{A}_{\ell}}(\vec{B}^{\SMG}_{\ell}\vec{A}_{\ell}) &= |\vec{B}^{\SMG}_{\ell}\vec{A}_{\ell}|_{\vec{A}_{\ell}}|(\vec{B}^{\SMG}_{\ell}\vec{A}_{\ell})^{-1}|_{\vec{A}_{\ell}} \eqreff{eq: SMG matrix symmetrization}= |\vec{I}- \vec{E}_{\ell}\vec{E}_{\ell}^{T_{\vec{A}_{\ell}}}|_{\vec{A}_{\ell}} |(\vec{I}- \vec{E}_{\ell}\vec{E}_{\ell}^{T_{\vec{A}_{\ell}}})^{-1}|_{\vec{A}_{\ell}} \\
		& \leq \frac{|1-\vec{E}_{\ell}\vec{E}^{T_{\vec{A}_{\ell}}}_{\ell}|_{\vec{A}_{\ell}}}{1 - |\vec{E}_{\ell}\vec{E}^{T_{\vec{A}_{\ell}}}_{\ell}|_{\vec{A}_{\ell}}} \leq  \frac{1 + |\vec{E}_{\ell}\vec{E}^{T_{\vec{A}_{\ell}}}_{\ell}|_{\vec{A}_{\ell}}}{1 - |\vec{E}_{\ell}\vec{E}^{T_{\vec{A}_{\ell}}}_{\ell}|_{\vec{A}_{\ell}}} \overset{\eqref{eq: norm equality}}{=} \frac{1 + |\vec{E}_{\ell}|^2_{\vec{A}_{\ell}}}{1 - |\vec{E}_{\ell}|^2_{\vec{A}_{\ell}}} \eqreff{eq: linearized MG contracion}{\leq}\frac{1 + q^2}{1-q^2}.
	\end{align*}
	This proves~\eqref{eq:sMG spectral bound}.
	Contraction~\eqref{eq:sMG contraction} follows from Proposition~\ref{theorem: CG convergence} and thus concludes the proof.
\qed

\section{Optimal additive Schwarz preconditioner}\label{sec: additive Schwarz operator}

In this section, we consider an additive Schwarz preconditioner $\vec{B}^{\AS}_{\ell}$ induced by the levelwise matrices $\vec{S}_{\ell'}$ from~\eqref{eq: levelwise matrices}. 

\subsection{Preconditioner and corresponding main result}
Define the additive Schwarz preconditioner $\vec{B}^{\AS}_{\ell}$ as
\begin{equation}\label{eq: AS preconditioner}
	\vec{B}^{\AS}_{\ell} \coloneqq \vec{I}_0^+ \vec{A}_0^{-1} (\vec{I}_0^+)^T + \sum_{\ell'=1}^{\ell-1} \vec{I}_{\ell'}^+ (\vec{D}_{\ell'}^+)^{-1} (\vec{I}_{\ell'}^+)^T + \sum_{z \in \VV_{\ell}} \vec{I}_{\ell,z}^p (\vec{A}_{\ell,z}^p)^{-1} (\vec{I}_{\ell,z}^p)^T.
\end{equation}
From its definition~\eqref{eq: AS preconditioner} and the definition~\eqref{eq: levelwise matrices} of the matrices $\vec{S}_{\ell'}$, it follows that $\vec{B}^{\AS}_{\ell}\vec{A}_{\ell} = \sum_{\ell'=0}^{\ell} \vec{S}_{\ell'}$.
We define the additive Schwarz operator $\SS_{\ell}^{\AS}$ via
\begin{equation}\label{eq: AS operator}
	\SS_{\ell}^{\AS} \coloneqq \sum_{\ell' =0}^\ell \SS_{\ell}.
\end{equation}

\begin{remark}[Computational complexity of additive Schwarz preconditioner]\label{remark: cost AS}
	Arguing as in Remark~\ref{rem: complexity of geometric MG}, we conclude that the overall computational cost of a single application of the additive Schwarz preconditioner $\vec{B}_{\ell}^{\AS}$ is $\OO(\#\TT_{\ell})$. 
    Moreover, since this is an additive method, in contrast to the multiplicative structure of $\vec{B}_{\ell}^{\MG}$ and $\vec{B}_{\ell}^{\SMG}$, its application can be efficiently parallelized, thereby reducing the overall computational time.
\end{remark}

The subsequent theorem shows that $\textnormal{cond}_{\vec{A}_{\ell}}(\vec{B}^{\AS}_{\ell}\vec{A}_{\ell})$ is bounded independently of $h$ and $p$. 
Thus, Proposition~\ref{theorem: CG convergence} yields that PCG with preconditioner $\vec{B}_{\ell}^{\AS}$ contracts $h$- and $p$-robustly.
The proof is postponed to Section~\ref{subsec: proof AS}.
Together with Remark~\ref{remark: cost AS}, it follows that $\vec{B}_{\ell}^{\AS}$ is optimal.

\begin{theorem}[PCG with additive Schwarz preconditioner]\label{theorem: AS condition number bound}
	The additive Schwarz preconditioner $\vec{B}_{\ell}^{\AS}$ defined in~\eqref{eq: AS preconditioner} is an SPD matrix.
	For the minimal and maximal eigenvalues of $\vec{B}^{\AS}_{\ell}\vec{A}_{\ell}$ it holds that
	\begin{equation*}
		c \leq \lambda_{\min}(\vec{B}^{\AS}_{\ell}\vec{A}_{\ell}) \quad \text{and} \quad \lambda_{\max}(\vec{B}^{\AS}_{\ell}\vec{A}_{\ell}) \leq C,
	\end{equation*}
	where $c,C>0$ are indepedent of $h$, $p$, and only depend locally on the diffusion contrast $\vec{K}$. 
	In particular, this yields
	\begin{equation*}
		\operatorname{cond}_2((\vec{B}_{\ell}^{\AS})^{1/2}\vec{A}_{\ell}(\vec{B}_{\ell}^{\AS})^{1/2}) = \operatorname{cond}_{\vec{A}_{\ell}}(\vec{B}^{\AS}_{\ell}\vec{A}_{\ell}) \leq C/c,
	\end{equation*}
	For the iterates $\vec{x}_{\ell}^k$ generated by PCG with preconditioner $\vec{B}_{\ell}^{\AS}$ and initial guess $\vec{x}_{\ell}^0$, there holds
	\begin{equation}\label{eq: AS contraction}
		|\vec{x}_{\ell}^\star- \vec{x}_\ell^{k+1}|_{\vec{A}_{\ell}} \leq \Big(1- \frac{c}{C}\Big)^{1/2} |\vec{x}_{\ell}^\star- \vec{x}_\ell^k|_{\vec{A}_\ell} \quad \text{for all } k \in \N_0.
	\end{equation}
	For $u_{\ell}^{\star} = \sum_{j=1}^{N_{\ell}^p} (\vec{x}_{\ell}^{\star})_j \varphi^p_{\ell,j}$ and $u_{\ell}^{k} = \sum_{j=1}^{N_{\ell}^p} (\vec{x}_{\ell}^{k})_j \varphi^p_{\ell,j}$, this directly translates to 
	\begin{equation*}
		 \enorm{u_{\ell}^{\star}- u_{\ell}^{k+1}} \leq \Big(1- \frac{c}{C}\Big)^{1/2} \, \enorm{u_{\ell}^{\star} - u_{\ell}^k} \quad \text{for all } k \in \N_0. 
	\end{equation*}
\end{theorem}

The following lemma provides a translation between functional and matrix notation for the additive Schwarz method, analogous to Lemma~\ref{lemma: connection between MG operator and matrix} for the geometric multigrid method.
The proof follows directly from Lemma~\ref{lemma: levelwise operator and Matrix} and is thus omitted.

\begin{lemma}[Additive Schwarz functional/matrix representation]
	Let $\vec{B}^{\AS}_{\ell}$ and $\SS_{\ell}^{\AS}$ denote the preconditioner and operator from~\eqref{eq: AS preconditioner}--\eqref{eq: AS operator}.
	Then, there holds 
	\begin{equation}\label{eq: connection AS operator and matrix}
		\edual{\SS_{\ell}^{\AS} v_{\ell}}{w_{\ell}} = (\vec{B}^{\AS}_{\ell}\vec{A}_{\ell}\vec{x}_{\ell},\vec{y}_{\ell})_{\vec{A}_{\ell}} \quad \text{for all } v_{\ell},w_{\ell} \in \XX_{\ell}^p \text{ and } \vec{x}_{\ell} = \chi^p_{\ell}(v_{\ell}), \, \vec{y}_{\ell} = \chi^p_{\ell}(w_{\ell}). \qed
	\end{equation} 
\end{lemma}

Similar to the multiplicative case, we show properties of the operator $\SS_{\ell}^{\AS} $ and translate them via~\eqref{eq: connection AS operator and matrix} to prove Theorem~\ref{theorem: AS condition number bound}.
For this purpose, we need some auxiliary results.

\subsection{Auxiliary results}\label{sec: AS auxiliary results}
First, recall Lions' lemma~\cite{Lio87}.

\begin{lemma}[Lions' lemma]\label{lemma: lion}
Let $\VV$ be a finite-dimensional Hilbert space with scalar product $\edual{\cdot}{\cdot}_{\VV}$ and 
corresponding norm $\enorm{\cdot}_{\VV}$. 
Suppose the decomposition $\VV = \sum_{j=0}^m \VV_{j}$ with corresponding orthogonal projections $\widetilde{\SS}_{j}: \VV \rightarrow \VV_{j}$ defined by
\begin{equation}
	\edual{\widetilde{\SS}_{j} v}{w_{j}}_{\VV} = \edual{v}{w_{j}}_{\VV} \quad \text{for all }v \in \VV \text{ and all } w_{j} \in \VV_{j}.
\end{equation}
If there exists a stable decomposition $v= \sum_{j=0}^m v_{j}$ with $v_{j} \in \VV_{j}$ for every $v \in \VV$ such that
\begin{equation}
	\sum_{j=0}^m \enorm{v_{j}}_{\VV}^2 \leq C \, \enorm{v}_{\VV}^2,
\end{equation}
then the operator $\widetilde{\SS} \coloneqq \sum_{j=0}^{m} \widetilde{\SS}_{j}$ satisfies, even with the same constant $C$,
\begin{equation}
	\enorm{v}_{\VV}^2 \leq C \edual{\widetilde{\SS} v}{v}_{\VV} \quad \text{for all } v \in \VV. \qed
\end{equation}
\end{lemma}

We also need a strengthened Cauchy--Schwarz inequality. 
This requires the use of generation constraints. 
For every $\ell' \in \{0, \dots, \ell\}$ and $z \in \VV_{\ell'}$, define the generation $g_{\ell',z}$ of $\TT_{\ell'}(z)$ by
\begin{equation}
	g_{\ell',z} \coloneqq \max_{T \in \TT_{\ell'}(z)} \level(T) \coloneqq \max_{T \in \TT_{\ell'}(z)} \log_2(|T_0|/|T|) \in \N_0,
\end{equation}  
where $T_0 \in \TT_0$ is the unique ancestor of $T$, i.e., $T \subseteq T_0$.

Let $M \coloneqq \max_{T \in \TT_{\ell}} \level(T)$. 
We denote by $\{\widehat{\TT}_j\}_{j=0}^M$ the sequence of uniformly refined triangulations that satisfy $\widehat{\TT}_{j+1} \coloneqq \mathtt{REFINE}(\widehat{\TT}_j,\widehat{\TT}_j)$ and $\widehat{\TT}_0\coloneqq \TT_0$. 
Since $\TT_0$ is admissible, each element $T \in \widehat{\TT}_j$ satisfies $\level(T)=j$ and hence is only bisected once during uniform refinement; see~\cite[Theorem~4.3]{Ste08}. 
Moreover, denote by $\widehat{h}_j\coloneqq \max_{T \in \widehat{\TT}_j} |T|^{1/d}$  the mesh-size of the uniform triangulation $\widehat{\TT}_j$. 
Importantly, there holds $\widehat{h}_j \simeq h_T$ for all $T \in \widehat{\TT}_j$ and all $j \in \N_0$ and the hidden constants depend only on $\TT_0$. 
Every object associated with uniform meshes will be indicated with a hat, e.g., $\widehat{\XX}_j^1$ is the lowest-order FEM space induced by $\widehat{\TT}_j$. 

\begin{lemma}\label{lemma: patches}
	Let $\ell' \in \{0, \dots, \ell\}$ and $z \in \VV_{\ell'}$. 
	Define $r_{\ell',z} \coloneqq \min_{T \in \TT_{\ell'}(z)} \level(T).$
	Then, there exist integers $\Cspan, n \in \N$ depending only on $\gamma$-shape regularity such that $m \coloneqq g_{\ell',z} \leq r_{\ell',z} + \Cspan$ and $\omega_{\ell'}(z) \subseteq \widehat{\omega}_{m}^n(z)$. 
\end{lemma}

\begin{proof}[\textbf{Proof}]
	The proof is split into two steps.

	\textbf{Step 1:} Let $z \in \VV_{\ell'}$. 
	Then, there exists elements $T,T' \in \TT_{\ell'}(z)$ such that $g_{\ell',z} = \level(T)$ and $r_{\ell',z} = \level(T')$. 
	Moreover, we have $|T| \simeq |T'|$ due to $\gamma$-shape regularity.
	Denote by $T_0, T_0' \in \TT_0$ the unique ancestors of $T$ and $T'$, respectively. 
	With the quasi-uniformity of $\TT_0$, it follows that
	\begin{equation*}
		\level(T) = \log_2 (|T_0|/|T|) \leq \log_2 (C \, |T_0'|/|T'|) = \log_2(C) + \level(T'),
	\end{equation*}
	where $C>0$ depends only $\gamma$-shape regularity.
	For $\Cspan \coloneqq \lceil \log_2(C)\rceil$, we get $m = g_{\ell',z} \leq r_{\ell',z} + \Cspan$.

	\textbf{Step 2:} By definition of $r_{\ell',z}$, we have $\omega_{\ell'}(z) \subseteq \widehat{\omega}_{r_{\ell',z}}(z)$. 
	Every element $T \in \widehat{\TT}_{r_{\ell',z}}$ can be decomposed into elements $T_j \in \widehat{\TT}_{r_{\ell',z}+\Cspan}$ with $j = 1, \ldots, 2^{\Cspan}$, i.e., $T = \bigcup_{j=1}^{2^{\Cspan}} T_j$.
	Thus, for every $T \subseteq \overline{\widehat{\omega}_{r_{\ell',z}}(z)}$ we also have $T \subseteq \overline{\widehat{\omega}_{r_{\ell',z}+ \Cspan}^{2^{\Cspan}}(z)}$.
	Hence, there exists an integer $n \in \N$ with $n \leq 2^{\Cspan}$ such that $\widehat{\omega}_{r_{\ell',z}}(z) \subseteq \widehat{\omega}^n_{r_{\ell',z}+\Cspan}(z)$.
	Finally, Step 1 yields
	\begin{equation*}
		\omega_{\ell'}(z) \subseteq \widehat{\omega}_{r_{\ell',z}}(z) \subseteq \widehat{\omega}_{r_{\ell',z}+\Cspan}^n(z) \subseteq \widehat{\omega}_{m}^n(z).
	\end{equation*}
	This concludes the proof.
\end{proof}

Let $\TT$ be a refinement of the initial mesh $\TT_0$ and $\MM \subseteq \TT$. 
For $\omega \coloneqq \operatorname{int}\bigl(\bigcup_{T \in \MM} T\bigr)$, we define
\begin{equation}\label{eq: local constant domain}
    C[\omega] \coloneqq \max \{\max_{T \in \MM}\|\div (\mathbf{K})\|_{L^{\infty}(T)}, \sup_{y \in \omega}\lambda_{\max}(\mathbf{K}(y))\}.
\end{equation}
From \cite[Lemma 5.6]{IMPS24} and~\cite{Pau25}, the following strengthened Cauchy--Schwarz inequality on uniform meshes is already known. 

\begin{lemma}[Strengthened Cauchy--Schwarz inequality]
	Let $0 \leq i \leq j$, $\widehat{\MM}_i \subseteq \widehat{\TT}_i$, and $\omega_i \coloneqq \operatorname{int}(\bigcup_{T \in \widehat{\MM}_i}T)$. 
	Then, it holds that
	\begin{equation}\label{eq: uniform strengthenedCS}
		\edual{\widehat{u}_i}{\widehat{v}_j}_{\widehat{\omega}_i} \leq \Cscs \, C[\widehat{\omega}_i] \, \delta^{j-i} \widehat{h}_j^{-1} 
		\norm{\nabla \widehat{u}_i}_{\widehat{\omega}_i} \norm{\widehat{v}_j}_{\widehat{\omega}_i} \quad \text{for all } \widehat{u}_i \in \widehat{\XX}_i^1  
		\text{ and } \widehat{v}_j \in \widehat{\XX}_j^1,
	\end{equation}
	where $\delta = 2^{-1/(2d)}$. 
	The constant $\Cscs$ depends only on $\Omega$, $d$, $\TT_0$, and $\gamma$-shape regularity. \qed
\end{lemma}

For $0 \leq m \leq M$, we define the operator $\GG_{\ell,m}^1: \XX_{\ell}^p \rightarrow \XX_{\ell}^{1}$ by
\begin{equation}\label{eq: G operator}
	\GG_{\ell,m}^1 \coloneqq \sum_{\ell'=1}^{\ell-1} \sum_{\substack{z \in \VV_{\ell'}^+ \\ g_{\ell',z} = m}} \PP_{\ell',z}^1.
\end{equation} 
It immediately follows that
\begin{equation}\label{eq: AS operator with generations}
	\SS_{\ell}^{\AS}  = \PP_0^1 + \sum_{m=0}^M \GG_{\ell,m}^1 + \sum_{z \in \VV_{\ell}} \PP_{\ell,z}^p.
\end{equation}
We show a strengthened Cauchy--Schwarz inequality for the operators $\GG_{\ell,m}^1$. 

\begin{lemma}[Strengthened Cauchy--Schwarz inequality for $\boldsymbol{\GG_{\ell,m}^1}$]\label{lemma: strengthened CS}
	For $0 \leq k \leq m \leq M$, it holds that
	\begin{equation}
		\edual{\widehat{v}_k}{\GG_{\ell,m}^1 \widehat{w}_k} \leq \Cloc \, \Cscs \, \delta^{m-k} \enorm{\widehat{v}_k} \: \enorm{\widehat{w}_k} 
		\quad \text{for all } \widehat{v}_k,\widehat{w}_k \in \widehat{\XX}_k^1,
	\end{equation}
	where the constant $\Cloc$ is defined in~\eqref{eq: local constant 2} and the constant $\Cscs>0$ depends only on $\Omega$, $d$, $\TT_0$,  and $\gamma$-shape regularity.
\end{lemma}

\begin{proof}[\textbf{Proof}]
	Let $0 \leq k \leq m \leq M$ and $\widehat{v}_k, \widehat{w}_k \in \widehat{\XX}_k^1$.
	The proof is split into two steps.

	\textbf{Step 1:}
	Let $\ell' \in \{1, \dots, \ell-1\}$ and $z \in \VV_{\ell'}^+$ with $g_{\ell',z} = m$. 
	Recall $C[\cdot]$ from~\eqref{eq: local constant domain}.
	Since $\PP_{\ell',z}^1 \widehat{w}_k \in \widehat{\XX}_{m,z}^1$, the strengthened Cauchy--Schwarz inequality~\eqref{eq: uniform strengthenedCS} implies
	\begin{equation*}
		\edual{\widehat{v}_k}{\PP_{\ell',z}^1 \widehat{w}_k}_T \overset{\eqref{eq: uniform strengthenedCS}}{\lesssim} 
		C[T]\delta^{m-k} \widehat{h}_m^{-1}\norm{ \nabla \widehat{v}_k}_T \: \norm{\PP_{\ell',z}^1 \widehat{w}_k}_T \quad \text{for all } T \in \widehat{\TT}_k.
	\end{equation*}
	As $\supp \PP_{\ell',z}^1 \widehat{w}_k \subseteq \overline{\omega_{\ell'}(z)}$, we only need to consider $T \in \widehat{\TT}_k$
	with $|T \cap \overline{\omega_{\ell'}(z)}| > 0$. 
	For these elements, we can find a vertex $z_0 \in \VV_0$ depending only on $z$ such that $T \cup \omega_{\ell'}(z) \subseteq \overline{\omega_0(z_0)}$. 
	Using the local norm equivalence, we obtain
	\begin{equation*}
		\norm{\nabla \widehat{v}_k}_T \leq \big(\inf_{y \in T} \lambda_{\min}(\vec{K}(y))\big)^{1/2} \, \enorm{\widehat{v}_k}_T.
	\end{equation*}
	Summation over $T \in \widehat{\TT}_k$, the existence of $z_0$, and the discrete Cauchy-Schwarz inequality yield
	\begin{align*}
		\edual{\widehat{v}_k}{\PP_{\ell',z}^1 \widehat{w}_k} &\lesssim C[\omega_{0}(z_0)] \big(\inf_{y \in \omega_0(z_0)} \lambda_{\min}(\vec{K}(y))\big)^{1/2}\delta^{m-k} \widehat{h}_{m}^{-1} \, \sum_{T \in \widehat{\TT}_k} \enorm{\widehat{v}_k}_T \norm{\PP_{\ell',z}^1 \widehat{w}_k}_T \\
		&\leq C[\omega_{0}(z_0)] \big(\inf_{y \in \omega_0(z_0)} \lambda_{\min}(\vec{K}(y))\big)^{1/2} \delta^{m-k} \widehat{h}_m^{-1}  \enorm{\widehat{v}_k} \norm{\PP_{\ell',z}^1 \widehat{w}_k}.
	\end{align*}
	The generation constraint $g_{\ell',z} = m$ and quasi-uniformity lead to $\widehat{h}_m \simeq h_{\ell',z}$. 
	Due to the local support of $\PP^1_{\ell',z}$, the Poincaré inequality, and the local norm equivalence, we obtain
	\begin{equation*}
		\widehat{h}_m^{-1} \norm{\PP_{\ell',z}^1 \widehat{w}_k} \simeq \widehat{h}_{\ell',z}^{-1} \norm{\PP_{\ell',z}^1 \widehat{w}_k}_{\omega_{\ell'}(z)} 
		\lesssim \norm{ \nabla \PP_{\ell',z}^1 \widehat{w}_k}
		\leq \big(\inf_{y \in \omega_0(z_0)} \lambda_{\min}(\vec{K}(y))\big)^{1/2} \: \enorm{\PP_{\ell',z}^1 \widehat{w}_k}
	\end{equation*}
	and thus, with $\Cloc$ from~\eqref{eq: local constant 2},
	\begin{equation}\label{eq: strengthened CS intermediate}
		\edual{\widehat{v}_k}{\PP_{\ell',z}^1 \widehat{w}_k} \lesssim \Cloc \, \delta^{m-k} \enorm{\widehat{v}_k} \: \enorm{\PP_{\ell',z}^1 \widehat{w}_k}.
	\end{equation}

	\textbf{Step 2:}
	Based on~\eqref{eq: strengthened CS intermediate} and the definition~\eqref{eq: G operator} of $\GG_{\ell,m}^1$, it only remains to prove that
	\begin{equation*}
		\sum_{\ell'=1}^{\ell-1} \sum_{\substack{z \in \VV_{\ell'}^+ \\ g_{\ell',z} = m}} \enorm{\PP_{\ell',z}^1 \widehat{w}_k} \lesssim \enorm{\widehat{w}_k}.
	\end{equation*}
	The definition~\eqref{eq: local projections} of $\PP_{\ell',z}^1$ implies $\PP_{\ell',z}^1 v = \frac{\edual{v}{\varphi^1_{\ell',z}}}{\enorm{\varphi^1_{\ell',z}}^2}\varphi^1_{\ell',z}$. 
	Since $g_{\ell',z}= m$, Lemma~\ref{lemma: patches} yields
	\begin{equation}\label{eq: estimate local projections}
		\enorm{\PP_{\ell',z}^1 \widehat{w}_k}  = \frac{|\edual{\widehat{w}_k}{\varphi^1_{\ell',z}}|}{\enorm{\varphi^1_{\ell',z}}} \leq \enorm{\widehat{w}_k}_{\omega_{\ell'}(z)} 
		\leq \enorm{\widehat{w}_k}_{\widehat{\omega}_m^n(z)}.
	\end{equation}
	Let $z \in \VV_{\ell}$ and $0 \leq j \leq M$.
	To keep track of the levels, where the patch associated to the vertex $z$ has been modified in the refinement and remains of generation $j$, we define 
	\begin{equation*}
		\fL_{\underline{\ell},\overline{\ell}}(z,j) \coloneqq \{\ell' \in \{\underline{\ell}, \dots, \overline{\ell}\} : z \in \VV_{\ell'}^+ \text{ and } g_{\ell',z} = j\} 
		\quad \text{for all } 0 \leq \underline{\ell} \leq \overline{\ell} \leq \ell.
	\end{equation*}
	According to~\cite[Lemma 3.1]{WC06}, there exists a constant $\Clev >0 $ depending only on $\gamma$-shape regularity such that 
	\begin{equation}\label{eq: uniform bound on levels}
		\max_{\substack{z \in \VV_{\ell} \\ 0 \leq j \leq M}} \# (\fL_{0,\ell}(z,j)) \leq \Clev < \infty.
	\end{equation} 
	Moreover, it holds that 
	\begin{align}\label{eq: change numbering}
		\begin{split}
		\{(\ell',z) \in \N_0 \times \VV_{\ell} : \ell' &\in \{\underline{\ell}, \dots, \overline{\ell}\}, z \in \VV_{\ell'}^+ \text{ with } g_{\ell',z} = j\} \\
		&=\{(\ell',z) \in \N_0 \times \VV_{\ell} : z \in \widehat{\VV}_j, \ell' \in \fL_{\underline{\ell}, \overline{\ell}}(z,j)\}.
		\end{split}
	\end{align}
	With finite patch overlap, this leads to
	\begin{align*}
		\sum_{\ell'=1}^{\ell-1} \sum_{\substack{z \in \VV_{\ell'}^+ \\ g_{\ell',z} = m}} \enorm{\PP_{\ell',z}^1 \widehat{w}_k} \, &\overset{\mathclap{\eqref{eq: estimate local projections}}}{\leq} \; \sum_{\ell'=1}^{\ell-1} \sum_{\substack{z \in \VV_{\ell'}^+ \\ g_{\ell',z} = m}} \enorm{\widehat{w}_k}_{\widehat{\omega}_m^n(z)} \overset{\eqref{eq: change numbering}}{=} \sum_{z \in \widehat{\VV}_m} \sum_{\ell' \in \LL_{1,\ell-1}(z,m)} \enorm{\widehat{w}_k}_{\widehat{\omega}_m^n(z)} \\ 
		&\overset{\eqref{eq: uniform bound on levels}}{\leq} \; \Clev \sum_{z \in \widehat{\VV}_m}  \enorm{\widehat{w}_k}_{\widehat{\omega}_m^n(z)} \lesssim \enorm{\widehat{w}_k}.
	\end{align*}
	Overall, we hence obtain
	\begin{equation*}
		\edual{\widehat{v}_k}{\GG_{\ell,m}^1 \widehat{w}_k} \lesssim \delta^{m-k} \enorm{\widehat{v}_k} \; \enorm{\widehat{w}_k}.
	\end{equation*}
	This concludes the proof.
\end{proof}

\subsection{Additive Schwarz operator}\label{sec: proof of optimality}
We require the following proposition on the additive Schwarz operator~\eqref{eq: AS operator} to finally show Theorem~\ref{theorem: AS condition number bound} in the next section.
For the proof of Proposition~\ref{theorem: AS operator}, we adapt~\cite{FFPE17} for the boundary element setting with energy space $\widetilde{H}^{1/2}(\Gamma)$ to our setting with energy space $H_0^1(\Omega)$.

\begin{proposition}[Properties of the additive Schwarz operator]\label{theorem: AS operator}
	The operator $\SS_{\ell}^{\AS} $ defined in~\eqref{eq: AS operator} is linear, bounded, and symmetric, i.e., there holds
	\begin{equation}\label{eq: AS operator symmetric}
		\edual{\SS_{\ell}^{\AS} v}{w} = \edual{v}{\SS_{\ell}^{\AS} w} \quad \text{for all } v,w \in \XX.
	\end{equation}
	Moreover, it holds that
	\begin{equation}\label{eq: AS operator bounds}
		c \: \enorm{v_{\ell}}^2 \leq \edual{\SS_{\ell}^{\AS}  v_{\ell}}{v_{\ell}} \leq C \: \enorm{v_{\ell}}^2 \quad \text{for all } v_{\ell} \in \XX_{\ell}^p,
	\end{equation}
	where the constants $c$ and $C$ depend only on $\Omega$, $\TT_0$, $d$, $\gamma$-shape regularity, $\Cloc$, and $\CLoc$.
\end{proposition}

\begin{proof}[\textbf{Proof}]
The proof is divided into three steps.

\textbf{Step 1 (basic properties):} The linearity, boundedness, and symmetry follow directly from the additive structure of $\SS_{\ell}^{\AS}$ and the respective property of the levelwise operators $\SS_{\ell'}$ from~\eqref{eq: levelwise operators}.

\textbf{Step 2 (lower bound in~(\ref{eq: AS operator bounds})):} We show the lower bound using Lemma~\ref{lemma: lion}. 
	For our application, we consider the space decomposition
	\begin{equation*}
		\XX_{\ell}^p = \XX_0^1 + \sum_{\ell'=1}^{\ell-1} \sum_{z \in \VV_{\ell'}^+} \XX_{\ell',z}^1 + \sum_{z \in \VV_{\ell}} \XX_{\ell,z}^p.
	\end{equation*}
	In~\cite[Proposition 5.5]{IMPS24} and~\cite{Pau25}, it is shown that there exists an $hp$-robust stable decomposition, i.e., for every $v_{\ell} \in \XX_{\ell}^p$, there exist functions $v_0 \in \XX_0^1$, $v_{\ell',z} \in \XX_{\ell',z}^1$, and $v_{\ell,z} \in \XX_{\ell,z}^p$ such that $v_{\ell} = v_0 + \sum_{\ell'=1}^{\ell-1}\sum_{z \in \VV_{\ell'}^+} v_{\ell',z}+ \sum_{z \in \VV_{\ell}}v_{\ell,z}$ and
	\begin{equation*}
		\enorm{v_0}^2 + \sum_{\ell'=1}^{\ell-1} \sum_{z \in \VV_{\ell'}^+} \enorm{v_{\ell',z}}^2 + \sum_{z \in \VV_{\ell}} \enorm{v_{\ell,z}}^2 \leq \widetilde{C} \, \enorm{v_{\ell}}^2,
	\end{equation*}
	where the constant $\widetilde{C}>0$ depends only on the initial triangulation $\TT_0$, $\gamma$-shape regularity, and $\CLoc$. 
	Therefore, Lemma~\ref{lemma: lion} yields the lower bound in~\eqref{eq: AS operator bounds} with $c \coloneqq 1/\widetilde{C}$.

\textbf{Step 2 (upper bound in~(\ref{eq: AS operator bounds})):} Let $v_{\ell} \in \XX_{\ell}^p$. 
	Recall $M = \max_{T \in \TT_{\ell}} \level(T)$ and observe $\GG_{\ell,m}^1 v_{\ell} \in \widehat{\XX}_m^1 \subseteq \widehat{\XX}_M^1$. 
	Utilizing~\eqref{eq: AS operator with generations}, we have 
	\begin{equation*}
		\edual{\SS_{\ell}^{\AS} v_{\ell}}{v_{\ell}} = \edual{\PP_0^1 v_{\ell}}{v_{\ell}} + \sum_{m=0}^{M} \edual{\GG_{\ell,m}^1 v_{\ell}}{v_{\ell}} +  \sum_{z \in \VV_{\ell}} \edual{\PP_{\ell,z}^p v_{\ell}}{v_{\ell}}.
	\end{equation*}
	To complete the proof, we must bound each of the summands by $\lesssim \enorm{v_{\ell}}^2$.

	On the initial level, we have $\edual{\PP_0^1 v_{\ell}}{v_{\ell}} \leq \enorm{v_{\ell}}^2$.

	On the finest level $\ell$, we apply the Cauchy--Schwarz inequality, Young inequality for $\mu>0$, and finite patch overlap to see
	\begin{align*}
		\sum_{z \in \VV_{\ell}} &\edual{\PP_{\ell,z}^p v_{\ell}}{v_{\ell}} \leq \enorm{v_{\ell}} \, \enormbig{\sum_{z \in \VV_{\ell}}\PP_{\ell,z}^p v_{\ell}} \leq \frac{\mu}{2} \enormbig{\sum_{z \in \VV_{\ell}} \PP_{\ell,z}^p v_{\ell}}^2 + \frac{1}{2\mu} \enorm{v_{\ell}}^2 \\
		&\leq  \frac{\mu}{2}(d+1) \sum_{z \in \VV_{\ell}} \, \enorm{\PP_{\ell,z}^p v_{\ell}}^2 + \frac{1}{2\mu} \enorm{v_{\ell}}^2  =  \frac{\mu}{2}(d+1) \sum_{z \in \VV_{\ell}} \edual{\PP_{\ell,z}^p v_{\ell}}{v_{\ell}} + \frac{1}{2\mu} \enorm{v_{\ell}}^2.
	\end{align*}
	With $\mu = (d+1)^{-1}$, this proves
	\begin{equation*}
		\sum_{z \in \VV_{\ell}} \edual{\PP_{\ell,z}^p v_{\ell}}{v_{\ell}} \leq (d+1) \, \enorm{v_{\ell}}^2.
	\end{equation*}

	It remains to consider the intermediate levels, where we will exploit the strengthened Cauchy--Schwarz inequality from Lemma~\ref{lemma: strengthened CS}.
	Let us denote by $\widehat{\QQ}_{m}: H_0^1(\Omega) \rightarrow \widehat{\XX}_m^1$ the Galerkin projections for the uniform meshes, i.e., $\edual{\widehat{\QQ}_m v}{\widehat{w}_m} = \edual{v}{\widehat{w}_m}$ for all $v \in H_0^1(\Omega)$ and all $\widehat{w}_m \in \widehat{\XX}_m^1$. 
	With $\widehat{\QQ}_{-1} \coloneqq 0$, it holds that
	\begin{equation*}
		\widehat{\QQ}_m = \sum_{k=0}^m (\widehat{\QQ}_k - \widehat{\QQ}_{k-1}).
	\end{equation*}
	Since the operators $\PP_{\ell',z}^1$ are symmetric, the bilinear form $\edual{\GG^1_{\ell,m} \cdot \,}{\cdot}$ is symmetric on the space $\XX_{\ell}^p$. 
	The positivity 
	\begin{equation*}
		\edual{\GG_{\ell,m}^1 v_{\ell}}{v_{\ell}} = \sum_{\ell'=1}^{\ell-1}\sum_{\substack{z \in \VV_{\ell'}^+ \\ g_{\ell',z} = m}} \edual{\PP_{\ell',z}^1 v_{\ell}}{v_{\ell}} 
		= \sum_{\ell'=1}^{\ell-1} \sum_{\substack{z \in \VV_{\ell'}^+ \\ g_{\ell',z}=m}} \enorm{\PP_{\ell',z}^1 v_{\ell}}^2 \geq 0 \quad \text{for all } v_{\ell} \in \XX_{\ell}^p
	\end{equation*}
	implies the Cauchy--Schwarz inequality
	\begin{equation}\label{eq: CS for operator G}
		\edual{\GG_{\ell,m}^1 v_{\ell}}{w_{\ell}} \leq \edual{\GG_{\ell,m}^1 v_{\ell}}{v_{\ell}}^{1/2} \edual{\GG_{\ell,m}^1 w_{\ell}}{w_{\ell}}^{1/2} \quad \text{for all } v_{\ell},w_{\ell} \in \XX_{\ell}^p.
	\end{equation}
	Consequently, we can estimate the sum over the intermediate levels by
	\begin{align*}
		\sum_{m=0}^{M} \edual{\GG_{\ell,m}^1 v_{\ell}}{v_{\ell}} &= \; \sum_{m=0}^{M} \edual{\GG^1_{\ell,m}v_{\ell}}{\widehat{\QQ}_m v_{\ell}} = \sum_{m=0}^M \sum_{k=0}^m \edual{\GG_{\ell,m}^1 v_{\ell}}{(\widehat{\QQ}_k - \widehat{\QQ}_{k-1})v_{\ell}} \\
		&\overset{\mathclap{\eqref{eq: CS for operator G}}}{\leq} \; \sum_{m=0}^M \sum_{k=0}^m \edual{\GG_{\ell,m}^1 v_{\ell}}{v_{\ell}}^{1/2} 
		\edual{\GG_{\ell,m}^1 (\widehat{\QQ}_k- \widehat{\QQ}_{k-1}) v_{\ell}}{(\widehat{\QQ}_k- \widehat{\QQ}_{k-1}) v_{\ell}}^{1/2}. 
	\end{align*}
	As $(\widehat{\QQ}_{k}-\widehat{\QQ}_{k-1})v_{\ell} \in \widehat{\XX}_k^1$, the strengthened Cauchy--Schwarz inequality from Lemma~\ref{lemma: strengthened CS} yields
	\begin{equation*}
		\edual{\GG_{\ell,m}^1 (\widehat{\QQ}_{k}-\widehat{\QQ}_{k-1})v_{\ell} }{(\widehat{\QQ}_{k}-\widehat{\QQ}_{k-1})v_{\ell}} \lesssim \delta^{m-k} \, \enorm{(\widehat{\QQ}_{k}-\widehat{\QQ}_{k-1})v_{\ell}}^2.
	\end{equation*}
	Moreover, we observe that
	\begin{align*}
		\enorm{(\widehat{\QQ}_k - \widehat{\QQ}_{k-1})v_{\ell}}^2 &= \edual{\widehat{\QQ}_k v_{\ell}}{ v_{\ell}}  
		- 2\edual{v_{\ell}}{\widehat{\QQ}_{k-1} v_{\ell}} + \edual{\widehat{\QQ}_{k-1} v_{\ell}}{ v_{\ell}} = \edual{(\widehat{\QQ}_k - \widehat{\QQ}_{k-1})v_{\ell}}{v_{\ell}}.
	\end{align*}
	Applying the Young inequality and the summability of the geometric series, we get
	\begin{align*}
		\sum_{m=0}^M \edual{\GG^1_{\ell,m} v_{\ell}}{v_{\ell}} &\lesssim  \sum_{m=0}^M \sum_{k=0}^m \delta^{(m-k)/2} \edual{\GG_{\ell,m}^1 v_{\ell}}{v_{\ell}}^{1/2} \edual{(\widehat{\QQ}_k-\widehat{\QQ}_{k-1})v_{\ell}}{v_{\ell}}^{1/2} \\
		&\leq \frac{\mu}{2} \sum_{m=0}^M \sum_{k=0}^{m} \delta^{m-k} \edual{\GG_{\ell,m}^1 v_{\ell}}{v_{\ell}} + 
		\frac{1}{2\mu} \sum_{m=0}^M \sum_{k=0}^m \delta^{m-k} \edual{(\widehat{\QQ}_k-\widehat{\QQ}_{k-1})v_{\ell}}{v_{\ell}} \\
		&\lesssim \frac{\mu}{2} \edualbig{\sum_{m=0}^{M}\GG^1_{\ell,m} v_{\ell}}{v_{\ell}} + \frac{1}{2 \mu} \sum_{k=0}^M \sum_{m=k}^M \delta^{m-k} \edual{(\widehat{\QQ}_k- \widehat{\QQ}_{k-1})v_{\ell}}{v_{\ell}} \\
		&\lesssim  \frac{\mu}{2} \edualbig{\sum_{m=0}^{M}\GG^1_{\ell,m} v_{\ell}}{v_{\ell}} + \frac{1}{2 \mu} \edualbig{\sum_{k=0}^M (\widehat{\QQ}_k-\widehat{\QQ}_{k-1})v_{\ell}}{v_{\ell}}.
	\end{align*}
	Since $\widehat{\QQ}_M$ is an orthogonal projection, it holds that
	\begin{equation*}
		\edualbig{\sum_{k=0}^M (\widehat{\QQ}_k-\widehat{\QQ}_{k-1})v_{\ell}}{v_{\ell}} = \edual{\widehat{\QQ}_M v_{\ell}}{v_{\ell}} \leq \enorm{v_{\ell}}^2.
	\end{equation*}
	Choosing $\mu$ sufficiently small, we obtain 
	\begin{equation*}
		\sum_{m=0}^M \edual{\GG^1_{\ell,m} v_{\ell}}{v_{\ell}} \lesssim \enorm{v_{\ell}}^2.
	\end{equation*}
	This concludes the proof.
\end{proof}

\subsection{Proof of Theorem~\ref{theorem: AS condition number bound}}\label{subsec: proof AS}
	By definition~\eqref{eq: AS preconditioner}, it is clear that $\vec{B}^{\AS}_{\ell}$ is an SPD matrix.
	Let $\vec{x}_{\ell},\vec{y}_{\ell} \in \R^{N_{\ell}^p}$, $v_{\ell}\coloneqq \sum_{j=1}^{N_{\ell}^p} (\vec{x}_{\ell})_j \varphi^p_{\ell,j} \in \XX_{\ell}^p$ and $w_{\ell}\coloneqq \sum_{j=1}^{N_{\ell}^p} (\vec{y}_{\ell})_j \varphi^p_{\ell,j} \in \XX_{\ell}^p$.
	Due to the identity~\eqref{eq: connection AS operator and matrix} and the symmetry of $\SS_{\ell}^{\AS} $, it follows that
	\begin{equation*}
		(\vec{B}^{\AS}_{\ell} \vec{A}_{\ell} \vec{x}_{\ell}, \vec{y}_{\ell})_{\vec{A}_{\ell}} \eqreff{eq: connection AS operator and matrix}{=} \edual{\SS_{\ell}^{\AS} v_{\ell}}{w_{\ell}} \eqreff{eq: AS operator symmetric}{=} \edual{v_{\ell}}{\SS_{\ell}^{\AS} w_{\ell}} \eqreff{eq: connection AS operator and matrix}{=} (\vec{x}_{\ell}, \vec{B}^{\AS}_{\ell}\vec{A}_{\ell} \vec{y}_{\ell})_{\vec{A}_{\ell}}.
	\end{equation*}
	Thus, the matrix $\vec{B}^{\AS}_{\ell}\vec{A}_{\ell}$ is symmetric with respect to the scalar product $(\cdot, \cdot)_{\vec{A}_{\ell}}$. 
	We use~\cite[Theorem C.1]{TW05} to write down the expressions for the maximal and minimal eigenvalue of $\vec{B}^{\AS}_{\ell}\vec{A}_{\ell}$.  
	Together with the identity~\eqref{eq: connection AS operator and matrix} and the bounds in~\eqref{eq: AS operator bounds}, we get
	\begin{equation*}
		\lambda_{\min}(\vec{B}^{\AS}_{\ell}\vec{A}_{\ell}) = \min_{\substack{\vec{x}_{\ell} \in \R^{N_{\ell}^p} \\ \vec{x}_{\ell} \neq 0}} \frac{(\vec{B}^{\AS}_{\ell}\vec{A}_{\ell}\vec{x}_{\ell},\vec{x}_{\ell})_{\vec{A}_{\ell}}}{(\vec{x}_{\ell},\vec{x}_{\ell})_{\vec{A}_{\ell}}} \overset{\eqref{eq: connection AS operator and matrix}}{=} \min_{\substack{v_{\ell} \in \XX_{\ell}^p \\ v_{\ell} \neq 0}} \frac{\edual{\SS_{\ell}^{\AS} v_{\ell}}{v_{\ell}}}{\enorm{v_{\ell}}^2} \overset{\eqref{eq: AS operator bounds}}{\geq} c
	\end{equation*}
	and
	\begin{equation*}
		\lambda_{\max}(\vec{B}^{\AS}_{\ell}\vec{A}_{\ell}) = \max_{\substack{\vec{x}_{\ell} \in \R^{N_{\ell}^p} \\ \vec{x}_{\ell} \neq 0}} \frac{(\vec{B}^{\AS}_{\ell}\vec{A}_{\ell}\vec{x}_{\ell},\vec{x}_{\ell})_{\vec{A}_{\ell}}}{(\vec{x}_{\ell},\vec{x}_{\ell})_{\vec{A}_{\ell}}} \overset{\eqref{eq: connection AS operator and matrix}}{=} \max_{\substack{v_{\ell} \in \XX_{\ell}^p \\ v_{\ell} \neq 0}} \frac{\edual{\SS_{\ell}^{\AS} v_{\ell}}{v_{\ell}}}{\enorm{v_{\ell}}^2} \overset{\eqref{eq: AS operator bounds}}{\leq} C.
	\end{equation*}
	Hence, we also obtain that
	\begin{equation*}
		\operatorname{cond}_{\vec{A}_{\ell}}(\vec{B}^{\AS}_{\ell}\vec{A}_{\ell}) = \frac{\lambda_{\max}(\vec{B}^{\AS}_{\ell}\vec{A}_{\ell})}{\lambda_{\min}(\vec{B}^{\AS}_{\ell}\vec{A}_{\ell})} \leq \frac{C}{c}.
	\end{equation*}
	Uniform contraction~\eqref{eq: AS contraction} follows directly from Proposition~\ref{theorem: CG convergence} and thus concludes the proof.
\qed

\section{Adaptive FEM with contractive solver}\label{sec: AFEM}

In this section, we present the application of uniformly contractive solvers in AFEM and the optimal complexity of the resulting adaptive algorithm.
Let us denote by $\Psi_{\ell}: \XX_{\ell}^p \rightarrow \XX_{\ell}^p$ the iteration map of an iterative solver of linear complexity which is uniformly contractive, i.e., there exists a constant $q \in (0,1)$ such that
\begin{equation}\label{eq: uniform contraction}
	\enorm{u_{\ell}^{\star}-\Psi_{\ell}(u_{\ell})} \leq q \, \enorm{u_{\ell}^{\star} - u_{\ell}} \quad \text{for all } u_{\ell} \in \XX_{\ell}^p.
\end{equation}
We consider the refinement indicators of the standard residual error estimator
\begin{subequations}\label{eq: refinement indicators}
\begin{equation}
	\eta_{\ell}(T,u_{\ell})^2 \coloneqq h_T^2 \|f+ \div(\vec{K}\nabla u_{\ell})\|^2_T + h_T \|\jump{\vec{K}\nabla u_{\ell}} \cdot \vec{n}\|^2_{\partial T \cap \Omega} \quad \text{for } T \in \TT_{\ell},
\end{equation}
where $\vec{n}$ denotes the outer normal vector of the element $T$ and $\jump{\cdot}$ denotes the jump across the element boundary.
Define 
\begin{equation}
	\eta_{\ell}(\UU_{\ell}, u_{\ell})^2 \coloneqq \sum_{T \in \UU_{\ell}} \eta_{\ell}(T,u_{\ell})^2 \quad \text{for } \UU_{\ell} \subseteq \TT_{\ell} \text{ and all } u_{\ell} \in \XX_{\ell}^p.
\end{equation}
\end{subequations}
For $\UU_{\ell} = \TT_{\ell}$, we abbreviate $\eta_{\ell}(u_{\ell})^2 \coloneqq \eta_{\ell}(\TT_{\ell}, u_{\ell})^2$.
Consider the adaptive algorithm with iterative solver from, e.g.,~\cite{GHPS21}.

\begin{algorithm}[AFEM with optimal iterative solver]\label{algo: AFEM}
	\textbf{Input:} Initial triangulation $\TT_0$, polynomial degree $p \geq 1$, initial guess $u_0^0 \coloneqq 0$, adaptivity parameters $0 < \theta \leq 1$, $\Cmark \geq 1$, and $\mu > 0$. 
	Then, for all $\ell=0,1,2,\dots$, perform the following steps~\eqref{item: AFEM 1}--\eqref{item: AFEM 3}:
	\begin{enumerate}[label=\textnormal{(\roman*)}, ref = \textnormal{\roman*}]
		\item \label{item: AFEM 1} \textbf{Solve \& Estimate:} For all $k=1,2,3, \dots,$ \textbf{repeat}~\eqref{item: AFEM solver step}--\eqref{item: AFEM stopping criterion} \textbf{until}
		\begin{equation*}
			\enorm{u_{\ell}^k-u_{\ell}^{k-1}} \leq \mu \, \eta_{\ell}(u_{\ell}^k):
		\end{equation*}
		\begin{enumerate}[label=\textnormal{(\alph*)}, ref = \textnormal{\alph*}]
			\item \label{item: AFEM solver step} Compute $u_{\ell}^k \coloneqq \Psi_{\ell}(u_{\ell}^{k-1})$ with one step of the algebraic solver.
			\item \label{item: AFEM stopping criterion} Compute the refinement indicators $\eta_{\ell}(T,u_{\ell}^k)$ for all $T \in \TT_{\ell}$.
		\end{enumerate}
		Upon termination of the $k$-loop, define the index $\underline{k}[\ell] \coloneqq k \in \N$ and $u_{\ell}^{\underline{k}} \coloneqq u_{\ell}^{k}$.
		\item \label{item: AFEM 2} \textbf{Mark:} Employ Dörfler marking to determine a set $\MM_{\ell} \in \M_{\ell}[\theta, u_{\ell}^{\underline{k}}] \coloneqq \{\UU_{\ell} \subset \TT_{\ell} : \theta \eta_{\ell}(u_{\ell}^{\underline{k}})^2  \leq \eta_{\ell}(\UU_{\ell}, u_{\ell}^{\underline{k}})^2 \}$ that fulfills
		\begin{equation*}
			\# \MM_{\ell} \leq \Cmark \min_{\UU_{\ell} \in \M_{\ell}[\theta, u_{\ell}^{\underline{k}}]} \# \UU_{\ell}.
		\end{equation*}
		\item \label{item: AFEM 3} \textbf{Refine:} Generate $\TT_{\ell+1} \coloneqq \textnormal{\texttt{refine}}(\TT_{\ell},\MM_{\ell})$ by newest vertex bisection and employ nested iteration $u_{\ell+1}^0 \coloneqq u_{\ell}^{\underline{k}}$.
		\end{enumerate}
	\textbf{Output:} Sequence of triangulations $\{\TT_{\ell}\}_{\ell \geq 0}$ and discrete approximations $\{u_{\ell}^{\underline{k}}\}_{\ell \geq 0}$.
\end{algorithm}

In order to formulate optimal complexity, we first define the countably infinite set
\begin{equation*}
	\QQ \coloneqq \{(\ell,k) \in \N_0^2 : u_{\ell}^k \text{ is defined in Algorithm~\ref{algo: AFEM}}\}.
\end{equation*}
The set $\QQ$ can be equipped with the natural order 
\begin{equation*}
	(\ell',k') \leq (\ell,k) :\Longleftrightarrow u_{\ell'}^{k'} \text{ is computed earlier than or equal to } u_{\ell}^{k} \text{ in Algorithm~\ref{algo: AFEM}}.
\end{equation*}
Furthermore, we define the total step counter by
\begin{equation*}
	|\ell,k| \coloneqq \# \{(\ell',k') \in \QQ : (\ell',k') \leq (\ell,k)\} \in \N_0 \quad \text{for } (\ell,k) \in \QQ.
\end{equation*}
Finally, we introduce the notion of approximation classes following~\cite{BDDP02,BDD04,Ste07,CKNS08,axioms}.
For any rate $s>0$, define
\begin{equation*}
	\|u^{\star}\|_{\A_s} \coloneqq \sup_{N \in \N_0} \big((N+1)^s \min_{\TT_{\textnormal{opt}} \in \T_N(\TT_0)} \eta_{\textnormal{opt}}(u^{\star}_{\textnormal{opt}})\big),
\end{equation*}
where $\T_N(\TT_0) \coloneqq \{\TT_H \in \T(\TT_0) : \# \TT_H - \# \TT_0 \leq N\}$.
The notation \(\TT_H \in \T(\TT_0)\) abbreviates that \(\TT_H\) can be obtained from \(\TT_0\) by a finite number of newest vertex bisection steps.
If $\|u^{\star}\|_{\A_s} < \infty$, then the error estimator $\eta_{\textnormal{opt}}(u_{\textnormal{opt}}^{\star})$ decays with rate $-s$ with respect to the number of elements of a sequence of (practically unknown) optimal triangulations $\TT_{\textnormal{opt}}$.
However, due to the iterative solver, rates of the adaptive algorithm should rather be considered with respect to the computational time or, equivalently, the total computational cost.
Hence, we comment on the computational cost of implementing Algorithm~\ref{algo: AFEM}:
Calcuting the error estimator $\eta_{\ell}(u_{\ell}^{k})$ has (up to quadrature) cost $O(\# \TT_{\ell})$ as this consist only of element-wise operations.
The marking step can be implemented with cost $O(\# \TT_{\ell})$; see~\cite{Ste07} for $\Cmark = 2$ and~\cite{PP20} for $\Cmark = 1$.
Finally, the cost of mesh-refinement by NVB is also $O(\# \TT_{\ell})$; see, e.g.,~\cite{Ste08, DSG23}.
If one step of the algebraic solver $\Psi_{\ell}$ is of linear complexity, then total computational cost to compute $u_{\ell}^k$ via Algorithm~\ref{algo: AFEM} is given by
\begin{equation}\label{eq: comp cost}
	\cost(\ell,k) \coloneqq \sum_{\substack{(\ell',k')\in \QQ \\ |\ell',k'| \leq |\ell,k|}} \# \TT_{\ell'}.
\end{equation}
We can now state optimal complexity of Algorithm~\ref{algo: AFEM}. 
The proof fits within the setting of~\cite[Theorem 2.3]{BFMPS23} relying on~\cite[Theorem 8]{GHPS21} and is thus omitted here.

\begin{theorem}[Optimal complexity of Algorithm~\ref{algo: AFEM}]\label{thm: optimal complexity}
	Let $s > 0$. 
	Let the algebraic solver $\Psi_{\ell}$ be given by either GPCG with $\vec{B}_{\ell}^{\MG}$ preconditioner~\eqref{eq: MG preconditioner} or PCG with $\vec{B}^{\SMG}_{\ell}$ preconditioner~\eqref{eq: sMG preconditioner} or $\vec{B}^{\AS}_{\ell}$  preconditioner~\eqref{eq: AS preconditioner} satisfying linear complexity and uniform contraction~\eqref{eq: uniform contraction}.
	Define the quasi-error by
	\begin{equation}\label{eq:optimal1}
		{\sf H}_{\ell}^k \coloneqq \enorm{u_{\ell}^{\star}- u_{\ell}^k} + \eta_{\ell}(u_{\ell}^{\star}) \quad \text{for all } (\ell,k) \in \QQ.
	\end{equation}
	For arbitrary $0 < \theta \leq 1$, $\Cmark\geq 1$ and $\mu > 0$, there holds full R-linear convergence, i.e., there exist constants $\Clin  \geq 1$ and $0 < \qlin < 1$ such that
	\begin{equation}\label{eq:optimal2}
		{\sf H}_{\ell'}^{k'} \leq \Clin \, \qlin^{|(\ell',k')| - |(\ell,k)|} {\sf H}_{\ell}^k \quad \text{for all } (\ell,k), (\ell',k') \in \QQ \text{ with } (\ell',k') \geq (\ell,k).
	\end{equation}
	With $\Ccost \coloneqq \Clin (1-\qlin^{1/s})^{-s},$ this yields
	\begin{equation}\label{eq:optimal3}
		\sup_{(\ell,k) \in \QQ} (\# \TT_{\ell})^s {\sf H}_{\ell}^k \leq \sup_{(\ell,k) \in \QQ} \cost(\ell,k)^s {\sf H}_{\ell}^k \leq \Ccost \sup_{(\ell,k) \in \QQ} (\# \TT_{\ell})^s {\sf H}_{\ell}^k \quad \text{for all } s > 0.
	\end{equation}
	Moreover, there exists $0< \theta^{\star} \leq 1$ and $\mu^{\star} > 0$ such that, for sufficiently small parameters
	\begin{equation}\label{eq:optimal4}
		0 < \mu < \mu^{\star} \quad  \text{and} \quad  0 < \frac{(\theta^{1/2} + \mu/\mu^{\star})^2}{(1-\mu/\mu^{\star})^2} < \theta^{\star}, 
	\end{equation}
	 Algorithm~\ref{algo: AFEM} guarantees that
	\begin{equation}\label{eq:optimal5}
		\copt \enorm{u^{\star}}_{\A_s} \leq \sup_{(\ell,k) \in \QQ} \cost(\ell,k)^s {\sf H}_{\ell}^k \leq \Copt \max\{\enorm{u^{\star}}_{\A_s}, {\sf H}_0^0\}.
	\end{equation}
	The constants $\Copt, \copt > 0$ depend only on the polynomial degree $p$, the initial triangulation $\TT_0$, the rate $s$, the adaptivity parameters $\theta$ and $\mu$, the solver contraction constant $q$, constants stemming from the axioms of adaptivity~\cite{axioms} for the residual error estimator~\eqref{eq: refinement indicators}, and the properties of NVB.
	In particular, every possible convergence rate is achieved with respect to the overall computational cost. \qed
\end{theorem}

The interpretation of Theorem~\ref{thm: optimal complexity} reads as follows: 
Unconditionally of the adaptivity parameters, Algorithm~\ref{algo: AFEM} leads essentially to contraction~\eqref{eq:optimal2} of the quasi-error~\eqref{eq:optimal1}, independently of the algorithmic decision for mesh-refinement or yet another solver step. 
This yields~\eqref{eq:optimal3}, which states that the convergence rate $-s$ with respect to the number of degrees of freedom $\textnormal{dim} (\XX^p_{\ell}) \simeq \# \TT_{\ell}$ and with respect to the total computational cost~\eqref{eq: comp cost} (and hence total computing time) coincide. 
Moreover, for sufficiently small parameters~\eqref{eq:optimal4}, Algorithm~\ref{algo: AFEM} will achieve optimal complexity~\eqref{eq:optimal5}: 
If the error estimator evaluated for the exact FE solutions decays with rate $-s$ along a sequence of optimal meshes, then Algorithm~\ref{algo: AFEM} guarantees that the quasi-error~\eqref{eq:optimal1} decays with rate $-s$ with respect to the computational cost, i.e., inexact solution does indeed not spoil the overall convergence behavior.

\begin{remark}
	Algorithm~\ref{algo: AFEM} and Theorem~\ref{thm: optimal complexity} can be generalized to non-symmetric second-order linear elliptic PDEs that fit into the framework of the Lax--Milgram lemma, see~\cite{FHMP}.
	Then, the algebraic solver employs a preconditioned GMRES method with an optimal symmetric and positive definite preconditioner for the principal part. Based on the present work, canonical preconditioning includes the linear and symmetric multigrid method from Section~\ref{sec: symmetric MG} as well as the multilevel additive Schwarz preconditioner from Section~\ref{sec: additive Schwarz operator}.
\end{remark}

\section{Numerical experiments}\label{sec: numerics}

In this section, we compare the behavior of the proposed algebraic solvers, where the multigrid solver from~\cite{IMPS24} is used as a baseline. 
For this, we investigate both the performance of the solver itself as well as its application in the adaptive Algorithm~\ref{algo: AFEM}. 
The experiments are conducted in the open-source Matlab package MooAFEM~\cite{IP23}.

\subsection{Considered model problems}  \label{subsec: model problems}
\begin{figure}
	\resizebox{\textwidth}{!}{
		\subfloat{
			\includegraphics{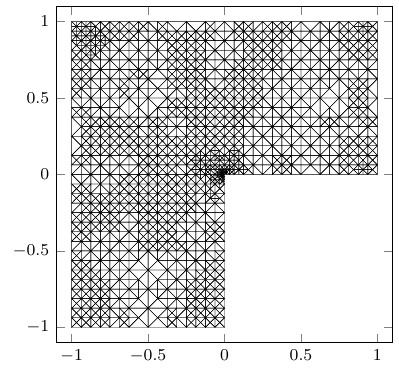}
		}
		\hfil
		\subfloat{
			\includegraphics{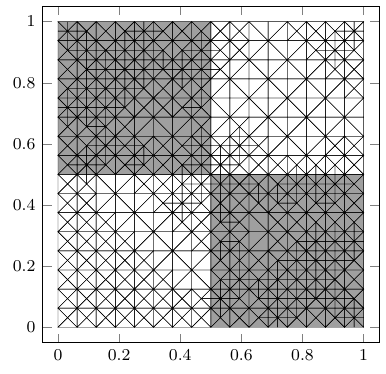}
		}
	}
	\caption{\label{fig: meshes} Adaptively refined meshes for the Poisson problem from Section~\ref{prob: lshape} with $\# \TT_{8} = 2490 $ (left) and checkerboard problem from Section~\ref{prob: checkerboard} with $\# \TT_{8}= 1242$ (right) for adaptivity parameters $\theta = 0.5$ and $\lambda = 0.1$.}
\end{figure}

We consider the following two test cases of model problem~\eqref{eq: model problem}:

\begin{itemize}
	\item \label{prob: lshape} \textit{Poisson}: Let $\Omega = (-1,1)^2 \backslash ([0,1] \times [-1,0])$ be the L-shaped domain with diffusion coefficient $\vec{K} = \vec{I}$ and right-hand side $f=1$. 
	An example of a mesh obtained via AFEM for this problem is displayed in Figure~\ref{fig: meshes} (left). 
	\item \label{prob: checkerboard} \textit{Checkerboard}: Let $\Omega = (0,1)^2$ be the unit square and $\vec{K}$ the $2 \times 2$ checkerboard diffusion with values 1 (white) and 100 (gray); see Figure~\ref{fig: meshes} (right). We refer to~\cite{Kel75} for an exact solution of this problem.
\end{itemize}

\subsection{Considered iterative solvers}  \label{subsec: iterative solvers}
We give an overview of all considered iterative solvers for the numerical experiments. 

\begin{itemize}
	\item \textbf{MG}: The geometric multigrid solver from~\cite{IMPS24} with $h$- and $p$-robust contraction factor; see Proposition~\ref{lemma: MG contraction}.
	\item \textbf{GPCG+MG}: GPCG of Algorithm~\ref{algo: GPCG} with the non-linear and non-symmetric multigrid preconditioner defined in~\eqref{eq: MG preconditioner}, for which Theorem~\ref{theorem: main result} establishes its $h$- and $p$-robust contraction.
	\item \textbf{PCG+AS}: PCG from Section~\ref{sec: GPCG} using the additive Schwarz preconditioner from~\eqref{eq: AS preconditioner}, for which Theorem~\ref{theorem: AS condition number bound} guarantees $h$- and $p$-robust contraction.
	\item \textbf{PCG+sMG}: PCG employing the symmetric multigrid preconditioner from~\eqref{eq: sMG preconditioner}, for which Theorem~\ref{corollary: symmetric MG optimality} ensures $h$- and $p$-robust contraction.
	\item \textbf{PCG+MG}: PCG with the non-linear and non-symmetric multigrid preconditioner defined in~\eqref{eq: MG preconditioner}. 
	While no theoretical contraction properties are known, we test experimentally the performance.
	\item \textbf{PCG+nsMG}: PCG with the non-symmetric multigrid preconditioner $\vec{B}_{\ell}^{\LMG}$ defined in Remark~\ref{remark: linearized MG}. 
	While no theoretical contraction properties are known, we test experimentally the performance.
\end{itemize}

\subsection{Solver contraction}  \label{subsec: solver contraction}

\begin{figure}
	\resizebox{\textwidth}{!}{
		\subfloat{
			\includegraphics{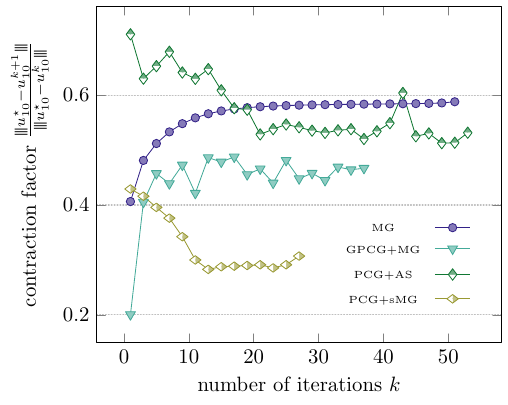}
		}
		\hfil
		\subfloat{
			\includegraphics{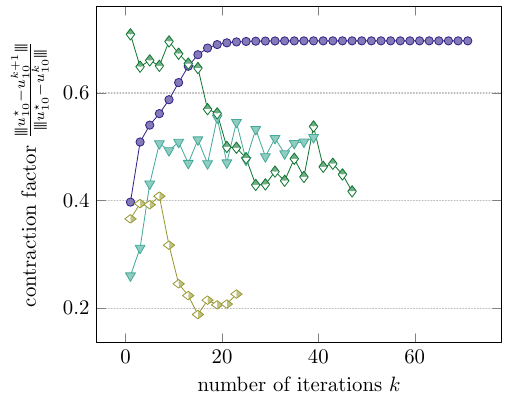}
		}
	}
	\caption{\label{fig: contraction optimal pre} History plot of the contraction factor for the Poisson problem from Section~\ref{subsec: model problems} with $\#\TT_{10} = 9723$ for $p=1$ (left) and $\# \TT_{10} = 340$ for $p=4$ (right) and solver stopping criterion $\enorm{u_{10}^{\star}-u_{10}^k} < 10^{-13}$.}
\end{figure}

To calculate the experimental contraction factors of the different methods, we first precompute a mesh hierarchy with $\ell = 10$ and corresponding polynomial degree for the Poisson problem from Section~\ref{subsec: model problems} using Algorithm~\ref{algo: AFEM} together with the geometric multigrid solver from~\cite{IMPS24} and parameters $\theta = 0.5$ and $\lambda = 0.1$. 
As stopping criterion, we iterate until the algebraic error $\enorm{u_{10}^{\star}-u_{10}^k}$ drops below $10^{-13}$; see
Figure~\ref{fig: contraction optimal pre}. 
While Theorem~\ref{theorem: main result} provides the same analytical contraction factor for both the standalone multigrid method and GPCG employing the multigrid preconditioner.
Figure~\ref{fig: contraction optimal pre} shows that the latter performs better.
Among the tested methods, PCG+sMG attains the smallest contraction factor.
However, one should interpret this comparison with caution, since each iteration of the symmetric preconditioner requires roughly twice the computational complexity of its non-symmetric counterpart.
Another notable difference is that the contraction factor of GPCG+MG increases with the number of solver steps, whereas for PCG methods it decreases.
This behavior can be attributed to PCG being a Krylov subspace method, where minimization occurs over an expanding subspace, while GPCG lacks this property.
This could also help explain why the contraction factor for PCG+sMG improves over the iterations.
Note that, as stated by theory, these contraction factors are indeed robust in the polynomial degree $p$. 

\subsection{Overall solver performance within AFEM}  \label{subsec: AFEM experiments}

\begin{figure}
	\resizebox{\textwidth}{!}{
		\subfloat{
			\includegraphics{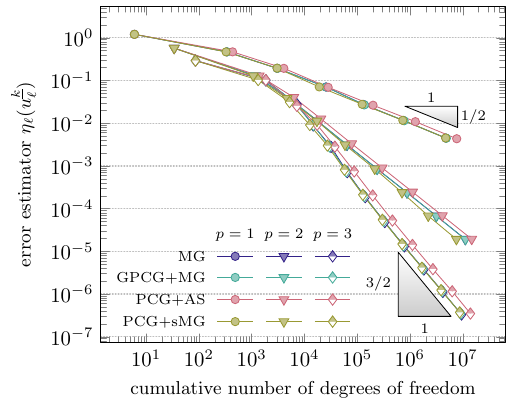}
		}
		\hfil
		\subfloat{
			\includegraphics{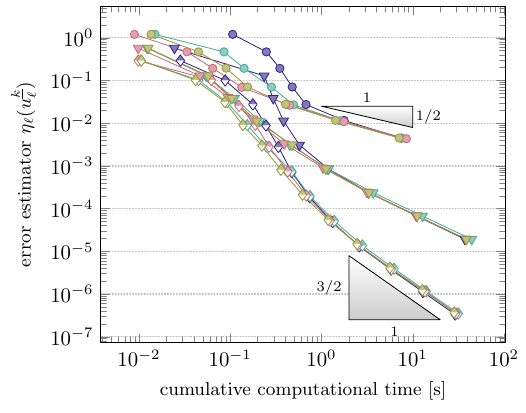}
		}
	}
	\caption{\label{fig: optimal cost pre} Convergence of the error estimator $\eta_{\ell}(u_{\ell}^{\underline{k}})$ for the Poisson problem from Section~\ref{subsec: model problems} with $\theta = 0.5$ and $\lambda = 0.05$.}
\end{figure}

\begin{figure}
	\resizebox{\textwidth}{!}{
		\subfloat{
			\includegraphics{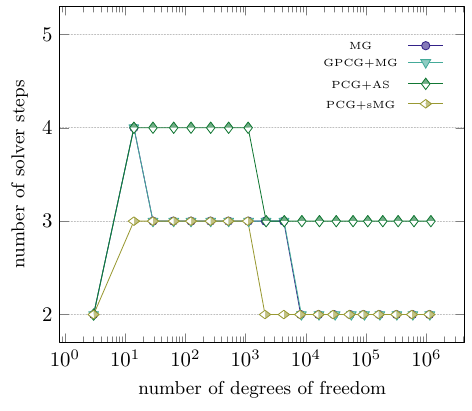}
		}
		\hfil
		\subfloat{
			\includegraphics{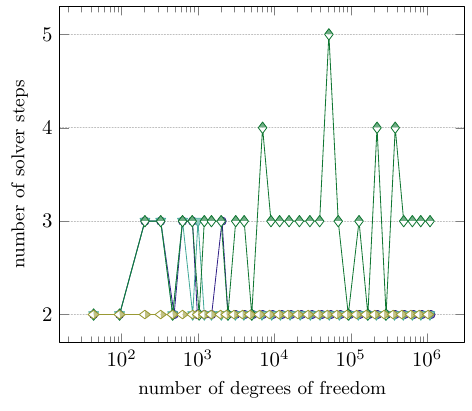}
		}
	}
	\caption{\label{fig: solver steps pre} History plot of solver steps with respect to the degrees of freedom for the Poisson problem from Section~\ref{subsec: model problems} with  $p=1$ (left) and $p=3$ (right).}
\end{figure}

Figure~\ref{fig: optimal cost pre} displays the error estimator $\eta_{\ell}(u_{\ell}^{\underline{k}})$, computed by Algorithm~\ref{algo: AFEM}, over the cumulative degrees of freedom and the cumulative time. 
Note that for the final iterate the error estimator $\eta_{\ell}(u_{\ell}^{\underline{k}})$ is equivalent to the quasi-error ${\sf H}_{\ell}^{\underline{k}}$ from Theorem~\ref{thm: optimal complexity}.
This follows from the axioms of adaptivity and the stopping criterion in Algorithm~\ref{algo: AFEM}; see, e.g.,~\cite{BMP24}.
After a pre-asymptotic phase, one can observe the optimal convergence rates $-p/2$ in Figure~\ref{fig: optimal cost pre} both with respect to the cumulative degrees of freedom and with respect to the cumulative time, across all optimal solvers GPCG+MG, PCG+AS, and PCG+sMG developed in this work.
In Figure~\ref{fig: optimal cost pre} (right), we see that the standalone MG is slower in the pre-asymptotic phase compared to the other solvers.
Moreover, Figure~\ref{fig: solver steps pre} confirms numerically that the number of solver steps with respect to the degrees of freedom for $p \in \{1,3\}$ remains uniformly bounded for all solvers.

\begin{figure}
	\resizebox{\textwidth}{!}{
		\subfloat{
			\includegraphics{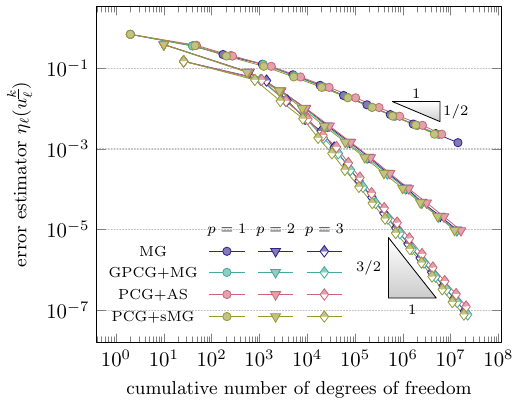}
		}
		\hfil
		\subfloat{
			\includegraphics{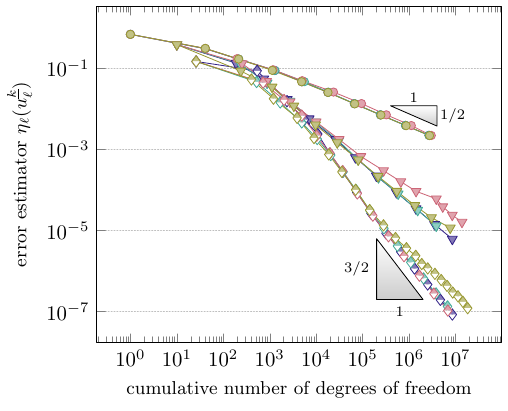}
		}
	}
	\caption{\label{fig: checkerboard optimality} Convergence of the error estimator for the checkerboard problem from Section~\ref{subsec: model problems} with $\lambda = 0.01$, $\theta = 0.3$ (left) and $\lambda = 0.05$, $\theta = 0.5$ (right).}
\end{figure} 

\begin{figure}
	\resizebox{\textwidth}{!}{
		\subfloat{
			\includegraphics{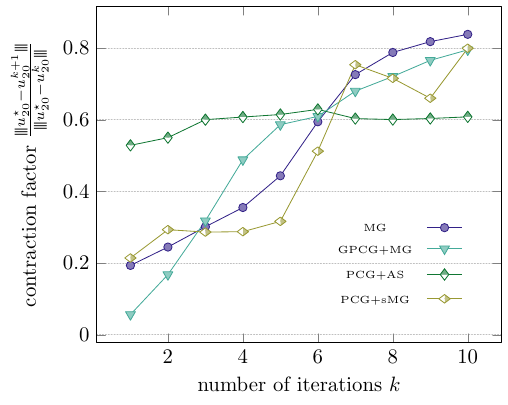}
		}
		\hfil
		\subfloat{
			\includegraphics{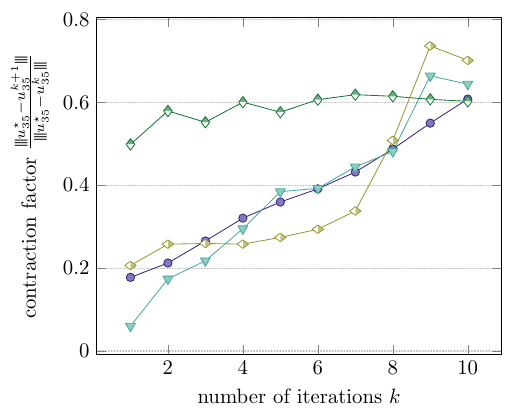}
		}
	}
	\caption{\label{fig: checkerboard contraction} History plot of the contraction factor for the checkerboard problem from Section~\ref{subsec: model problems} with $\ell=20$ (with $\# \TT_{\ell} = 97136$) for $p=2$  (left) and $\ell=35$ (with $\# \TT_{\ell} = 77681$) for $p=3$ (right).}
\end{figure}

We now consider the checkerboard problem from Section~\ref{subsec: model problems} for AFEM with MG, GPCG+MG, PCG+AS, and PCG+sMG.
The parameters are set to $\lambda \in \{0.01,0.05\}$, $p \in \{1,3\}$, and $\theta \in \{0.3,0.5\}$, and we study the decrease of the error estimator $\eta_{\ell}(u_{\ell}^{\underline{k}})$ with respect to the cumulative number of degrees of freedom.
In Figure~\ref{fig: checkerboard optimality} (left), all four solvers exhibit optimal convergence rates for $\lambda = 0.01$ and $\theta = 0.3$. 
In contrast, when $\lambda = 0.05$ and $\theta = 0.5$, the convergence rates for PCG+sMG degrade for both $p=2$ and $p=3$. 
Additionally, PCG+AS shows a suboptimal rate for $p=2$. 
While MG and GPCG+MG appear to maintain optimal rates across all configurations, a closer inspection for $p=3$ reveals a slightly suboptimal rate, as shown in Figure~\ref{fig: checkerboard optimality} (right).
This does not contradict Theorem~\ref{thm: optimal complexity}, since the assumptions require sufficiently small adaptivity parameters.

To gain a deeper insight into the solvers' behavior, we compute their experimental contraction factors using a pre-computed mesh hierarchy with $\ell = 20$ levels for $p=2$ and $\ell=35$ levels for $p=3$. 
Figure~\ref{fig: checkerboard contraction} shows the first 10 iterations, where PCG+AS and PCG+sMG exhibit larger contraction factors in the initial iterations compared to GPCG+MG. 
For this specific problem and parameter choices ($\lambda = 0.01$ and $\theta = 0.3$), the number of iterations per level required by any tested algebraic solver within the AFEM algorithm never exceeded 8 and was typically around 2. 
This observation may partially explain the inferior performance of PCG+AS and PCG+sMG. 
Additionally, both MG and GPCG+MG employ optimal step sizes (see step~\eqref{eq: MG step 2} and~\eqref{eq: MG step 3} of Algorithm~\ref{algo: geometric MG}) within the multigrid scheme, whereas PCG+sMG relies on a fixed step size (see Algorithm~\ref{algo: symmetric MG}), which may further contribute to its suboptimal behavior.

\begin{figure}
	\resizebox{\textwidth}{!}{
		\subfloat{
			\includegraphics{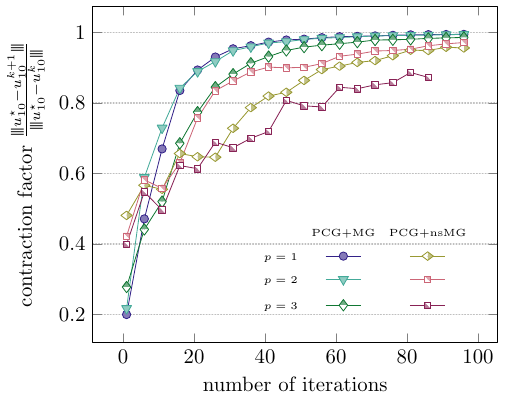}
		}
		\hfil
		\subfloat{
			\includegraphics{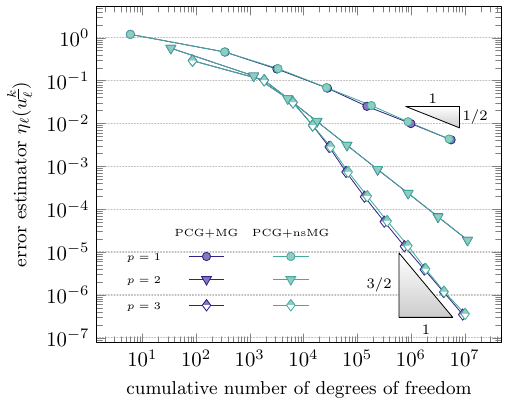}
		}
	}
	\caption{\label{fig: bad preconditoners contraction} History plot of the contraction factor for the Poisson problem from Section~\ref{subsec: model problems} employing the algebraic solvers PCG+MG and PCG+nsMG until $\enorm{u_{10}^{\star}-u_{10}^k} < 10^{-13}$ or for a maximum of 100 iterations with $\# \TT_{10} = 9723$ for $p=1$, $\# \TT_{10} = 1096$ for $p=2$, and $\#\TT_{10} = 405$ for $p=3$ (left).
	Convergence of the error estimator $\eta_{\ell}(u_{\ell}^{\underline{k}})$ for AFEM with PCG+MG and PCG+nsMG using $\theta = 0.5$ and $\lambda = 0.05$ (right).}
\end{figure}

\subsection{Warning example}
Although the theory does not cover non-linear and/or non-symm- \linebreak etric preconditioners within PCG, we nevertheless test the solvers PCG+MG and PCG+nsMG from Section~\ref{subsec: iterative solvers} numerically.
Figure~\ref{fig: bad preconditoners contraction} (left) shows that the contraction factors deteriorate for PCG+MG for all polynomial degree and for PCG+nsMG when $p \in \{1,2\}$. 
Since, in these cases the algebraic error $\enorm{u_{10}^{\star}-u_{10}^k}$ does not drop below $10^{-13}$, we stop the solver  after 100 iterations and report the final errors in Table~\ref{tab: final error}.
Despite this poor algebraic performance, we also test PCG+MG and PCG+nsMG within AFEM. 
Figure~\ref{fig: bad preconditoners contraction} (right) shows optimal convergence rates
with respect to the cumulative number of degrees of freedom. 
This is likely because AFEM typically requires only a small number of solver iterations per level (see Figure~\ref{fig: solver steps pre}).
Consequently, the deterioration of the contraction factors has little influence in this setting. 
However, for more difficult problems, or for smaller values of the parameter $\mu$, AFEM requires more solver steps, so that optimal complexity will be affected by the degraded contraction behavior. 
Thus, both PCG+MG and PCG+nsMG are not suitable for use within AFEM.

\begin{table}
	\centering
	\begin{tabular}{ccccc}
		\toprule
		& \multicolumn{2}{c}{PCG+MG} & \multicolumn{2}{c}{PCG+nsMG} \\
		\cmidrule(r){2-3} \cmidrule(r){4-5}
		& \footnotesize final error & \# \footnotesize solver steps & \footnotesize final error & \footnotesize \# solver steps  \\
		\midrule
		$p=1$ & $1.373\times 10^{-7}$ & 100 & $4.562 \times 10^{-12}$ & 100 \\
		$p=2$  & $5.816 \times 10^{-7}$ & 100 & $7.677 \times 10^{-10}$ & 100 \\
		$p=3$ & $1.127 \times 10^{-9}$ & 100 & $9.338 \times 10^{-14}$ & 90\\
		\bottomrule
	\end{tabular}
	\vspace{0.1cm}
	\caption{\label{tab: final error} Final algebraic error and number of solver steps of PCG+MG and PCG+nsMG for the Poisson problem from Section~\ref{subsec: model problems} applied to a pre-computed mesh hierarchy with $\ell=10$ levels after 100 iterations or the first error below $10^{-13}$. We stress that convergence of PCG+MG and PCG+nsMG remains theoretically open and can fail.}
\end{table}

\printbibliography

\end{document}